\documentclass[11pt]{article}  
\usepackage{amsmath}
\usepackage{amssymb}
\usepackage{theorem}
\usepackage{euscript}
\usepackage{pstricks}
\usepackage{authblk}
\usepackage{verbatim}
\usepackage{hyperref}

\usepackage{ifthen}
\usepackage{xifthen}
\usepackage{xcolor}
\usepackage{fullpage}
\usepackage[english]{babel}
\usepackage{mathtools}
\usepackage{hyperref}
\usepackage{color}

\usepackage{hyperref}
\hypersetup
{
	colorlinks,
	citecolor=black,
	filecolor=black,
	linkcolor=black,
	urlcolor=magenta
}
\hypersetup{linktocpage}
\newtheorem{theorem}{Theorem}[section]
\newtheorem{lemma}[theorem]{Lemma}

\numberwithin{equation}{section}

\theoremstyle{definition}
\newtheorem{definition}[theorem]{Definition} 
\theoremstyle{remark}

\newcommand{\brac}[1]{\left(#1\right)}

\newcommand{\set}[2]{\{#1\,:\,#2\}}

\newcommand{\bb}{{\boldsymbol{b}}}

\newcommand{\be}{{\boldsymbol{e}}}

\newcommand{\bell}{{\boldsymbol{\ell}}}

\newcommand{\bs}{{\boldsymbol{s}}}

\newcommand{\bx}{{\boldsymbol{x}}}

\newcommand{\by}{{\boldsymbol{y}}}

\newcommand{\bz}{{\boldsymbol{z}}}

\newcommand{\bW}{{\boldsymbol{W}}}

\newcommand{\brho}{{\boldsymbol{\rho}}}
\newcommand{\bvarrho}{{\boldsymbol{\varrho}}}
\newcommand{\bsigma}{{\boldsymbol{\sigma}}}
\newcommand{\bphi}{{\boldsymbol{\phi}}}

\newcommand{\rd}{{\rm d}}

\def\II{\mathbb I}

\def\ZZ{{\mathbb Z}}

\def\RR{{\mathbb R}}
\def\RRd{{\mathbb R}^d}

\def\NN{{\mathbb N}}
\def\N{{\mathbb N}}
\def\NNd{{\NN}^d}

\def\II{{\mathbb I}}

\def\CC{{\mathbb C}}

\def\NN{{\mathbb N}}
\def\R{{\mathbb R}}
\def\RR{{\mathbb R}}

\def\Vv{{\mathbb P}}
\def\UU{{\mathbb U}}
\def\Vv{{\mathcal P}}

\def\FF{{\mathcal F}}
\def\Bb{{\mathcal B}}

\def\NNd{{\mathbb N}^d}
\def\RRd{{\mathbb R}^d}
\def\RRm{{\mathbb R}^m}

\def\IIm{{\mathbb I}^m}

\def\IIi{{\mathbb I}^\infty}

\def\RRi{{\mathbb R}^\infty}

\def\RRm{{\mathbb R}^m}

\def\UUi{{\mathbb U}^{\infty}}

\def\Bb{{\mathcal B}}

\def\Ll{{\mathcal L}}
\def\Oo{{\mathcal O}}

\def\Ss{{\mathcal S}}

\def\Vv{{\mathcal V}}

\def\II{{\mathbb I}}
\def\CC{{\mathbb C}}
\def\ZZ{{\mathbb Z}}
\def\NN{{\mathbb N}}
\def\RR{{\mathbb R}}

\def\FF{{\mathbb F}}

\def\NNd{{\mathbb N}^d}
\def\RRd{{\mathbb R}^d}

\def\RRi{{\mathbb R}^\infty}

\def\supp{\operatorname{supp}}
\def\dv{\operatorname{div}}

\newcommand{\norm}[2]{\left\|{#1}\right\|_{#2}}

\newcommand{\abs}[1]{\left|#1\right|}

\title{\sffamily {Deep ReLU neural network approximation in  Bochner spaces and applications to  
 parametric PDEs}}

\author[a]{Dinh D\~ung \footnote{Corresponding author}}
\affil[a]{Information Technology Institute, Vietnam National University, Hanoi
	\protect\\
	144 Xuan Thuy, Cau Giay, Hanoi, Vietnam
	\protect\\
	Email: dinhzung@gmail.com}

\author[b]{Van Kien Nguyen}
\affil[b]{Faculty of Basic Sciences, University of Transport and Communications
	\protect\\	No.3 Cau Giay Street, Lang Thuong Ward, Dong Da District,
	Hanoi, Vietnam
	\protect\\
	Email: kiennv@utc.edu.vn}

\author[c]{Duong Thanh Pham}
\affil[c]{Vietnamese German University,
		\protect\\ Ring road 4, Quarter 4, Thoi Hoa Ward, Ben Cat Town, Binh Duong Province, Vietnam
	\protect\\
	Email: duong.pt@vgu.edu.vn}

\date{\today}
\begin{document}
\maketitle

\begin{abstract}
We investigate  non-adaptive methods of deep ReLU neural network approximation in Bochner spaces $L_2({\mathbb U}^\infty, X, \mu)$ of  functions on ${\mathbb U}^\infty$ taking values in a separable Hilbert space $X$, where ${\mathbb U}^\infty$ is either ${\mathbb R}^\infty$ equipped with the standard Gaussian probability measure, or ${\mathbb I}^\infty:= [-1,1]^\infty$ equipped with the Jacobi probability  measure. Functions to be approximated are assumed to satisfy a certain weighted $\ell_2$-summability of the generalized chaos polynomial expansion coefficients with respect to the measure $\mu$.
We prove  the convergence rate of this approximation in terms of the size of  approximating deep ReLU neural networks. These results then are applied to approximation of 
the solution  to  parametric elliptic PDEs with random inputs for the lognormal and affine cases. 

	
	
\end{abstract}


\section{Introduction}

 The aim of the present paper is to  construct deep ReLU neural networks for approximation of functions in Bochner spaces and to apply the obtained results to approximation of solutions to
 parametric and stochastic elliptic PDEs with lognormal or affine inputs. We investigate the convergence rate of this approximation in terms of the size of  approximating deep ReLU neural networks.
 
 The universal approximation capacity of neural networks has been known since the 1980's (\cite{Cyben89,HSW89, Funa90,Bar91}).
  In recent years, deep neural networks have been rapidly developed and 
  successfully applied to a wide range of fields.  The main advantage of deep neural networks over shallow ones is that they can output compositions of functions cheaply. 
 Since their application range is getting wider, theoretical analysis revealing reasons of these significant
practical improvements attracts substantial attention  \cite{ ABMM17,DDF.19,MPCB14,Te15,Te16}. 
 In the last several years, there has been a number of interesting papers that addressed the role of depth and architecture of deep neural networks in approximating functions   that possess  special  regularity properties such as analytic functions \cite{EWa18,Mha96}, differentiable functions \cite{PeVo18,Ya17a}, oscillatory functions \cite{GPEB21}, functions in  Sobolev or Besov spaces \cite{AlNo20,GKNV21,GKP20,Ya18}. High-dimensional approximations by deep neural networks have been studied in \cite{MoDu19, Suzu18,DN20}, and  their applications to high-dimensional PDEs in  \cite{SS18,EGJS18,OSZ21,GS20,GS21,GH21}.  Most of  these papers used  deep   ReLU  (Rectified Linear Unit)  neural networks  since the rectified linear unit is a simple and preferable activation function in many applications. The output of such  neural networks is  continuous piece-wise linear function which are easily and cheaply  computed. We refer the reader to the recent surveys  \cite{DHP21,Pet20} for various problems and aspects of neural network approximation and bibliography.
 
In computational uncertainty quantification, the problem of efficient  numerical approximation for  parametric and stochastic partial differential equations (PDEs) has been of great interest and achieved  significant progress in recent years.  There is a  vast number of works on this topic to not mention all of them. We  point out just some works \cite{BCDC17,BCM17,BCDM17,CoDe15a,CDS10,CDS11,Dung19,Dung21,EST18,HoSc14,ZDS19} which are directly related to our paper.  

Recently,  a number of works have been devoted to various problems and methods of  deep neural network approximation for parametric and stochastic PDEs with affine inputs on the compact set $\IIi:= [-1,1]^\infty$ and with the condition on uniform ellipticity, such as dimensionality reduction \cite{TB18}, deep neural network expression rates for the Taylor generalized polynomial chaos expansion  (gpc) of solutions to parametric elliptic PDEs \cite{SZ19}, reduced basis methods \cite{KPRS21} the problem of learning the discretized parameter-to-solution map in practice \cite{GPR.21}, Bayesian PDE inversion \cite{HSZ20}, etc.

Let $D \subset \RR^d$ be a bounded  Lipschitz domain.  Consider the diffusion elliptic equation 
\begin{equation} \label{ellip}
- \dv (a\nabla u)
\ = \
f \quad \text{in} \quad D,
\quad u|_{\partial D} \ = \ 0, 
\end{equation}
for  a given fixed right-hand side $f$ and a 
spatially variable scalar diffusion coefficient $a$.
Denote by $V:= H^1_0(D)$ the energy space and $H^{-1}(D)$ the dual space of $V$. Assume that  $f \in H^{-1}(D)$ (in what follows this preliminary assumption always holds without mention). If $a \in L_\infty(D)$ satisfies the ellipticity assumption
\begin{equation} \nonumber
0<a_{\min} \leq a \leq a_{\max}<\infty,
\end{equation}
by the well-known Lax-Milgram lemma, there exists a unique 
solution $u \in V$  to the equation~\eqref{ellip}  in the weak form
\begin{equation} \nonumber
\int_{D} a\nabla u \cdot \nabla v \, \rd \bx
\ = \
\langle f , v \rangle,  \quad \forall v \in V.
\end{equation}
Partial differential equations (PDEs) with parametric and stochastic inputs are a common model used  in science
	and engineering.  Depending on the 
	nature of the modeled object, the parameters involved in them  may be 
	either deterministic or random. Radom nature reflects the uncertainty in various parameters
	presented in the physical phenomenon modelled  by the equation.
	For the equation~\eqref{ellip}, 
	we consider  the diffusion coefficients having a parametric form $a=a(\by)$, where $\by=(y_j)_{j \in \NN}$
is a sequence of real-valued parameters ranging in the set 
$\UUi$ which is either  $\RRi$ or $\IIi$.
Denote by $u(\by)$ the solution to the 
parametric  diffusion elliptic equation 
\begin{equation} \label{p-ellip}
- {\rm div} (a(\by)\nabla u(\by))
\ = \
f \quad \text{in} \quad D,
\quad u(\by)|_{\partial D} \ = \ 0. 
\end{equation}	
The resulting solution operator maps
	$\by\in \UUi$ to $u(\by)\in V$. The objective is to 
achieve a numerical 
approximation of this complex map by a small number of parameters with a
guaranteed error in a given norm. 

In the present paper, we consider the lognormal case when $\UUi= \RRi$ and the diffusion coefficient $a$  is of the form
\begin{equation} \label{lognormal}
a(\by)=\exp(b(\by)), \quad {\text{with }}\ b(\by)=\sum_{j = 1}^\infty y_j\psi_j,
\end{equation}
 and $y_j$ are i.i.d. standard Gaussian random 
 variables, and the affine  case when $\UUi= \IIi$ and the diffusion coefficient $a$  is of the form
\begin{equation} \label{affine}
	a(\by)= \bar a + \sum_{j = 1}^\infty y_j\psi_j.
\end{equation}
Here $\bar a\in L_\infty(D)$ and $\psi_j \in L_\infty(D)$ for both the cases.

 Let us briefly describe the problem setting and main contribution of  the present paper.  

 We are interested in the problem of deep ReLU neural network approximation of the solution $u(\by)$ to  parametric and stochastic elliptic PDEs \eqref{p-ellip} with  lognormal inputs  \eqref{lognormal}
 or affine inputs \eqref{affine}.

  Recently, for  parametric  elliptic PDEs \eqref{p-ellip} with random inputs, we have investigated related problems of  sparse-grid polynomial interpolation approximation \cite{Dung21}, and of  deep ReLU neural network approximation based on polynomial interpolations \cite{Dung22}.
 A unified approach suggested in \cite{Dung21,Dung22}  to those problems, is to study non-adaptive methods of  sparse-grid polynomial interpolation and  deep ReLU neural network approximations  in genenal Bochner spaces which are able to be applied to parametric and stochastic elliptic PDEs. 
Following this approach, in the present paper, we 
 construct  non-adaptive methods of deep ReLU neural network approximation in Bochner spaces $L_2(\UUi, X, \mu)$ of  functions on $\UUi$ taking values in a separable Hilbert space $X$, where $\UUi$ is either $\RRi$ equipped with the standard Gaussian probability measure $\mu= \gamma$, or $\IIi:= [-1,1]^\infty$ equipped with the Jacobi probability  measure $\mu= \nu_{a,b}$.  Differing from the deep ReLU neural network approximation in \cite{Dung22}  which relies on polynomial interpolations, the approximation of  functions $ v \in L_2(\UUi, X, \mu)$ in the present paper is based  on an $m$-term truncation 
 $$
 S_m v(\by) := \sum_{j=1}^m v_{\bs^j} P_{\bs^j}(\by)
 $$
  of the orthonormal generalized polynomial chaos  (gpc) expansion  with respect to the measure $\mu$:
 \begin{equation} \label{series-P_bs}
 	v(\by)=\sum_{\bs\in\FF} v_\bs \,P_\bs(\by), \quad v_\bs \in X,
 \end{equation}
 where $\FF$ is  the set of all families of non-negative integers $\bs=(s_j)_{j \in \NN}$ with  finite support.
 Functions $v$ to be approximated are assumed to satisfy a certain weighted $\ell_2$-summability of the coefficients $(v_\bs)_{\bs \in \FF}$ of the expansion \eqref{series-P_bs}. 
For every integer $n \ge 4$, we construct  a non-adaptive  deep ReLU neural network $\bphi_n =(\phi_{\bs^j})_{j=1}^m$  on $\RRm$ with $m = \Oo(n/\log n)$, having size not greater than $n$ and $m$ outputs so that the approximation of $v$ by the function 
$$
\Phi_n (\by):= \sum_{j=1}^m v_{\bs^j} \phi_{\bs^j}(\by)
$$ 
gives the twofold  error bounds: 
 \begin{equation} \label{ErrorBnd}
\norm{v - \Phi_n}{L_2(\UUi, X, \mu)}  =	 \Oo\brac{m^{-1/q}}= \Oo\brac{\left(n/\log n\right)^{-1/q}}.
 \end{equation}
Here, the error bounds are much sharper than those estabilished in \cite{Dung22} for  deep ReLU neural network approximation based on polynomial interpolations. 

 These results are then applied to approximation in the space $L_2(\UUi, V, \mu)$ of the solution $u(\by)$ to  parametric and stochastic elliptic PDEs \eqref{p-ellip} with the lognormal inputs \eqref{lognormal}
 and the affine inputs \eqref{affine}. 
 They are also applied to the approximation of holomorphic functions on $\RRi$ in the present paper,  but we believe that these results  are applicable to holomorphic functions on $\IIi$.

  We give some comments on the difference of our contribution from the directly related paper \cite{SZ19}.  
  In \cite{SZ19}, the authors proved an  error bound  $\Oo\brac{n^{-(1/q - 1/2)}/(\log n \, \log \log n)}$ of the uniform approximation (or equivalently, the approximation in the norm of Bochner space 
  $L_\infty(\IIi, V, \lambda)$ with the uniform Lebesgue $\lambda$)  by deep ReLU neural networks of size $n$ of solution to parametric elliptic PDE \eqref{p-ellip} with affine inputs \eqref{affine}, based on  the  non-orthogonal Taylor  gpc expansion.  
In the present paper, we improved this result as $\Oo\brac{n^{-(1/q - 1/2)}/(\log n)}$ and extended this approximation to the Bochner space  $L_2(\UUi, V, \nu_{a,b})$  with the Jacobi probability  measure, based 
 the orthonormal Jacobi  gpc expansion. The  error bound of this extended approximation is as in the right-hand side of \eqref{ErrorBnd}.
  The most principal difference  of the present paper from  \cite{SZ19} is the results on the approximation in the norm of Bochner space 
 $L_2(\RRi, V, \gamma)$  by deep ReLU neural networks of size $n$ of solution to parametric elliptic PDE \eqref{p-ellip}  on non-compact set $\RRi$ with lognormal inputs \eqref{lognormal}, based on  the  orthonormal Hermite  gpc expansion. The construction of approximating deep ReLU neural networks and proof of the bound for the approximation error are much more technically complicated.
   Our main idea of constructing the approximating deep ReLU neural networks  $\bphi_n$ as well as the proof of  the error  bound \eqref{ErrorBnd} is that we first approximate the solution $u \in L_2(\RRi, V, \gamma)$   by the truncation $S_m u$ with error bound $\Oo(m^{-1/q})$. Then we construct a deep ReLU neural network  $\bphi_n$ on finite-dimensional space $\RRm$ that approximate $S_m u$ with error $\Oo(m^{-1/q})$.  Finally  we estimate  the size of the deep neural networks $\bphi_n$ as $n = \Oo\brac{m\log m}$. The construction and proofs of the bounds are based on the weighted $\ell_2$-summabilty of the gpc Hermite expansion coefficients of  $u$, the known realization of approximate multiplication by deep ReLU neural networks, and some known results on Gaussian-weighted polynomial approximation on $\RRm$.  The last but not least difference is, as noticed  all of the results of the present paper are deduced from a general theory of deep ReLU neural network in Bochner spaces.

Notice that the error bound of this deep ReLU neural network approximation is comparable to the error bound  of  the best  approximation of $v$ by   (linear non-adaptive) $n$-term truncations of the gpc expansion as well as of the approximation by the particular truncation $S_n v$ which is $\Oo(n^{-1/q})$ \cite{BCDM17,Dung21}. However, a deep ReLU neural network
 represents a continuous piece-wise linear  function defined on
	a number of polyhedral subdomains and therefore, is easily generated and computed in numerical implementation. For some PDEs, it has been shown that deep neural networks are capable of representing
	solutions without incurring the curse of dimensionality, see, for instance, \cite{JSW18,GJS19,HJKN19}. Therefore  deep ReLU neural networks are one of the  preferable tool in numerical solving (parametric) PDEs. We refer the reader to \cite{GPR.21,KPRS21,SZ19} for further discussion on the role of neural networks in the approximation
	of solutions to parametric PDEs.

In the present paper, we are concerned about  the parametric approximability for parametric and stochastic elliptic PDEs. Therefore, the results  themselves do not yield a practically realizable approximation since they do not cover the approximation of the gpc expansion coefficients which are functions of the spatial variable.  Moreover, the approximant $\Phi_n$, as we can see, is not a real deep ReLU networks, but just a combination of the gpc coefficients and the components of a deep ReLU network.  Naturally, it would be desirable to study the problem of full neural network approximation of the solution to  parametric and stochastic elliptic PDEs by combining the spatial and parametric domains based on fully discrete approximation  in \cite{BCDC17,Dung21}. We will discuss this problem in a forthcoming paper. 

The paper is organized as follows. 
In Section \ref{Deep ReLU neural networks}, we present a necessary knowledge  about deep ReLU neural networks.
Section \ref{Approximation in Bochner space with Gaussian measure}  is devoted to the investigation of non-adaptive methods of deep ReLU neural network approximation in Bochner space with Gaussian measure and applications to approximations  of the solution to  the 
parameterized  diffusion elliptic equation  \eqref{p-ellip} with lognormal inputs \eqref{lognormal}, and of  holomorphic functions on $\RRi$. In Section \ref{Approximation in Bochner space with Jacobi measure}, we study non-adaptive methods of deep ReLU neural network approximation in Bochner space with Jacobi measure and applications to approximations  of the solution to  the 
parameterized  diffusion elliptic equation  \eqref{p-ellip} with affine inputs \eqref{affine}. 

\medskip
\noindent
{\bf Notation} \  As usual, $\NN$ denotes the natural numbers, $\ZZ$  the integers, $\RR$ the real numbers and $ \NN_0:= \{s \in \ZZ: s \ge 0 \}$.
We denote  $\RR^\infty$ the
set of all sequences $\by = (y_j)_{j\in \NN}$ with $y_j\in \RR$. Denote by $\FF$  the set of all sequences of non-negative integers $\bs=(s_j)_{j \in \NN}$ such that their support $\nu_{\bs}:= \supp (\bs):= \{j \in \NN: s_j >0\}$ is a finite set. We use $(\be^j)_{j\in \NN}$ for the standard basis of $\ell_2(\NN)$.
For a set $G$, we denote by
$|G|$ the cardinality of $G$. 
We use letters $C$  and $K$ to denote general 
positive constants which may take different values, and 
$C_{\alpha,\beta,\ldots}$ and $K_{\alpha,\beta,\ldots}$  when we 
want to emphasize the dependence of these  constants on 
$\alpha,\beta,\ldots$, or when this dependence is 
important in a particular situation.

\section{Deep ReLU neural networks}
\label{Deep ReLU neural networks}

In this section, we present some necessary definitions and 
elementary facts on deep ReLU neural networks.
There is a wide variety of neural network architectures and each of them is adapted to specific tasks. As in \cite{Ya17a}, we will use  deep feed-forward neural networks which
  allow  connections between neurons  in a layer with neurons in any preceding  layers (but not in the same layer). This improves the bound of the  weight of parallelization network via the weights of the component networks (see Lemma \ref{lem:parallel} below) in comparing with the standard deep feed-forward neural networks which allows the connection between neurons only in neighboring layers (comp. \cite{SZ19}).
 
 In deep neural network approximation, we will employ the ReLU activation function that is defined by 
 $\sigma(t):= \max\{t,0\}$, $t\in \RR$.  We will use the notation 
 $\sigma(\bx):= (\sigma(x_1),\ldots, \sigma(x_d))$ for $\bx=(x_1,\ldots,x_d) \in \RRd$.

\begin{definition}\label{def:DNN}
	Let $d,L\in \NN$, $L\geq 2$, $N_0=d$,  and $N_1,\ldots,N_{L}\in \NN$. Let 
	$\bW^\ell=\brac{w^\ell_{i,j}} \in \RR^{N_\ell\times \brac{\sum_{i=0}^{\ell-1}N_i}}$, $\ell=1,\ldots,L$, be 
	$N_\ell\times \brac{\sum_{i=0}^{\ell-1}N_i}$ matrices, and $\bb^\ell =(b^\ell_j)\in \RR^{N_\ell}$.    A  ReLU neural network  $\Phi$ with input dimension $d$, output dimension  $N_L$ and $L$ layers
		is  a sequence of matrix-vector tuples
		$$
		\Phi=\brac{(\bW^1,\bb^1),\ldots,(\bW^L,\bb^L)},
		$$
		in which the following
	computation scheme is implemented
	\begin{align*}
		\bz^0&: = \bx \in \RR^d;
		\\
		\bz^\ell &: = \sigma\brac{\bW^{\ell}\brac{\bz^0,\ldots,\bz^{\ell-1}}^{{\rm T}}+\bb^\ell}, \ \ \ell=1,\ldots,L-1;
		\\
		\bz^L&:= \bW^L{\brac{\bz^0,\ldots,\bz^{L-1}}^{{\rm T}}} + \bb^L.
	\end{align*}
	We call $\bz^0$ the input  and with an 
	ambiguity denote 
	$\Phi(\bx):= \bz^L$ the output of $\Phi$  and in some places we identify a  ReLU neural network  with its output.
		We will use the following terminology.
		\begin{itemize}
			\item The number of layers $L(\Phi)=L$  is the depth of $\Phi$;
			\item The number of nonzero $w^\ell_{i,j}$ and $b^\ell_j$  is the  size of $\Phi$ and denoted by $W(\Phi)$;
			\item When $L(\Phi) \ge 3$, $\Phi$ is called a deep neural network, and otherwise, a shallow neural network.
		\end{itemize}
\end{definition}

The following two results are easy to verify from the definition above. We also refer the reader to \cite[Remark 2.9 and Lemma 2.11]{GKP20} for further remarks and comments.

\begin{lemma}[Parallelization]\label{lem:parallel}
	Let $N\in \NN$,  $\lambda_j\in \RR$, $j=1,\ldots,N$. Let $\Phi_j$, $j=1,\ldots,N$ be deep ReLU neural networks with input dimension $d$. Then  we can explicitly construct a deep ReLU neural network  denoted by $\Phi$ so that 
	$$
	\Phi(\bx)
	=
\begin{pmatrix}
\Phi_1(\bx)\\
\ldots\\
\Phi_N(\bx)
\end{pmatrix}
,\quad \bx\in \RR^d.
	$$ 
	Moreover, we have 
	$$
	W(\Phi){\le}\sum_{j=1}^NW(\Phi_j)
 \qquad \text{and}\qquad
	L(\Phi)= \max_{j=1,\ldots,N} L(\Phi_j).
	$$
	 The network $\Phi$ is called the parallelization network of $\Phi_j$, $j=1,\ldots,N$ and denoted by $\Phi=(\Phi_j)_{j=1}^N$.
\end{lemma}

\begin{lemma}[Concatenation]\label{lem-concatenation} Let $\Phi_1$ and $\Phi_2$ be two ReLU neural networks   such that output layer of $\Phi_1$ has the same dimension as input layer of $\Phi_2$. Then, we can explicitly construct a ReLU neural network $\Phi$ such that $\Phi(\bx)=\Phi_2(\Phi_1(\bx))$ for $\bx\in \RR^d$. Moreover we have
	$$
	W(\Phi)\leq W(\Phi_1)+W(\Phi_2) \qquad \text{and}\qquad L(\Phi) = L(\Phi_1)+L(\Phi_2).
	$$
The network $\Phi$ is called the concatenation network of $\Phi_1$ and $\Phi_2$.
\end{lemma}

The following lemma is easily deduced from \cite[Propositions 3.1 and 3.3]{SZ19} and \cite[Proposition 2]{MoDu19}.

	\begin{lemma} \label{lem:multi}
		Let $\bell \in \NNd$. For  every $\delta \in (0,1)$, $d\in \NN$, $d\geq 2$, we can explicitly construct  a deep  ReLU neural network  $\Phi_P$  so that 
		$$
		\sup_{ \bx \in [-1,1]^d} \Bigg|\prod_{j=1}^d x_j^{\ell_j} - \Phi_P(\bx) \Bigg| \leq \delta. 
		$$
		Furthermore, if $x_j=0$ for some $j\in \{1,\ldots,d\}$ then $\Phi_P(\bx)=0$ and there exists a constant $C>0$ independent of $\delta$ and $d$ such that
		$$
	W(\Phi_P) \leq C(1+ |\bell |_1\log (|\bell |_1\delta^{-1})) 
		\qquad \text{and}\qquad
		L(\Phi_P)  \leq C(1+\log  |\bell |_1\log(|\bell |_1\delta^{-1})) \,.
		$$
	\end{lemma}
	
For $j=0,1$, let $\varphi_j$ be the continuous piece-wise {linear} functions with break points $\{-2, -1,1,2\}$  such that $\supp(\varphi_j) = [-2,2]$, $\varphi_0(x)=1$ and $\varphi_1(x)=x$ if $x\in [-1,1]$.
Notice that $\varphi_j $ can be realized exactly by a shallow  ReLU neural network (still denoted by  $\varphi_j$), i.e., 
$$
\varphi_0(x)= \sigma(x-2) - 3\sigma(x-1) + 4\sigma(x) - 3\sigma(x+1) + \sigma(x+2), 
$$
and 
$$
\varphi_1(x)= \sigma(x-2) - 2\sigma(x-1) +  2\sigma(x+1) -\sigma(x+2).
$$

	\begin{lemma} \label{lem-product-varphi}
	Let $\bell \in \NNd$.	Let $\varphi$ be either $\varphi_0$ or $\varphi_1$.
		For  every $\delta \in (0,1)$, $d\in \NN$, we can explicitly construct  a deep  ReLU neural network  $\Phi$  so that 
		$$
		\sup_{ \bx \in [-2,2]^d} \Bigg|\prod_{j=1}^d\varphi^{\ell_j}(x_j) - \Phi(\bx) \Bigg| \leq \delta. 
		$$
		Furthermore, $\supp(\Phi )\subset [-2,2]^d$ and there exists a constant $C>0$ independent of $\delta$ and $d$ such that
\begin{equation}\label{eq-varphi}
		W(\Phi) \leq C\big(1+ |\bell |_1\log (|\bell |_1\delta^{-1}) \big)
\qquad \text{and}\qquad
L(\Phi)  \leq C\big(1+\log |\bell |_1\log(|\bell |_1\delta^{-1})) \,.
\end{equation}
	\end{lemma}

	\begin{proof}
		This simple lemma is proven in \cite{Dung22}. For completeness, let us recall the proof.
		The network $\Phi$ is constructed as a concatenation of $\{\varphi(x_j)\}_{j=1}^d$ and $\Phi_P$. Notice that
		$$
		\{\varphi(x_j)\}_{j=1}^d \subset [-1,1]^d, \ \ \forall \bx \in [-2,2]^d.
		$$
		 Hence, the estimate \eqref{eq-varphi} follows directly from Lemmata  \ref{lem-concatenation} and \ref{lem:multi} and the estimates $W(\varphi_0)\leq 10$, $W(\varphi_1)\leq 8$.
\hfill
	\end{proof}

\section{Approximation in Bochner spaces with Gaussian measure}
\label{Approximation in Bochner space with Gaussian measure}

In this section, we investigate  non-adaptive methods of deep ReLU neural network approximation  in the Bochner space $L_2(\RRi,X,\gamma)$ with Gaussian measure $\gamma$ of functions   on $\RRi$ taking values in a separable Hilbert space $X$ and satisfying a weighted $\ell_2$-summability  of the Hermite gpc expansion coefficients. We
construct such methods and prove convergence rates of the approximation by them. These methods are constructed via the truncations on finite-dimensional supercube of the truncated Hermite gpc expansion  of functions in $L_2(\RRi,X,\gamma)$ to be approximated.
The results are then applied to the
approximation of the solution  to the parametric elliptic PDEs \eqref{p-ellip} with lognormal inputs \eqref{lognormal} on  $\RRi$ and of holomorphic maps on $\RRi$.

\subsection{Approximation by truncations of the Hermite gpc expansion}
\label{Approximation by truncations of the Hermite gpc expansion}

We first recall a concept of infinite tensor product 
of probability measures. 
Let $\mu(y)$ be a probability measure on $\UU$.
We introduce the probability measure $\mu(\by)$ on $\UUi$ as 
the infinite tensor product of the probability 
measures $\mu(y_i)$:
\begin{equation} \nonumber
\mu(\by) 
:= \ 
\bigotimes_{j \in \NN} \mu(y_j) , \quad \by = (y_j)_{j \in \NN} \in \UUi.
\end{equation}
If $\varrho (y)$ is the density of $\mu(y)$, i.e., $\rd \mu(y) = \varrho(y) \rd y $, then we  write
\begin{equation} \nonumber
\mbox{d} \mu(\by) 
:= \ 
\bigotimes_{j \in \NN} \varrho(y_j) \rd (y_j), \quad \by = (y_j)_{j \in \NN} \in \UUi.
\end{equation}
(For details on infinite tensor product of probability measures, see, e.g., \cite[pp. 429--435]{HS65}.)

Let $X$ be a separable Hilbert space. 
The probability measure $\mu(\by)$ induces  the Bochner space $L_2(\UUi,X,\mu)$ of $\mu$-measurable mappings $v$ from $\UUi$ to $X$ for which the norm 
\begin{equation} \nonumber
\|v\|_{L_2( \UUi,X,\mu)}
:= \
\left(\int_{ \UUi} \|v(\cdot,\by)\|_X^2 \, \mbox{d} \mu(\by) \right)^{1/2} \ < \ \infty.
\end{equation}

We consider approximations in the space $L_2(\RRi,X,\gamma)$  with   
 the infinite tensor product  standard Gaussian probability measure $\gamma$ on $\RRi$. More precisely,
let $\gamma(y)$ be the probability measure on $\RR$ 
with the standard Gaussian density: 
\begin{equation*} \label{g}
\rd\gamma(y):=g(y)\,\rd y, \quad g(y):=\frac 1 {\sqrt{2\pi}} e^{-y^2/2}.
\end{equation*}
Then the infinite tensor product standard Gaussian 
probability measure $\gamma(\by)$ on $\RRi$ can be 
defined by
\begin{equation} \nonumber
\rd \gamma(\by) 
:= \ 
\bigotimes_{j \in \NN} g(y_j) \rd (y_j).
\end{equation}
We make use of the abbreviation  
$\Ll_2(X):=L_2(\RRi,X,\gamma)$. 
	
A powerful strategy for 
approximation of functions $v \in \Ll_2(X)$ is based on  
truncations of the Hermite gpc expansion
\begin{equation} \label{series}
v(\by)=\sum_{\bs\in\FF} v_\bs \,H_\bs(\by), \quad v_\bs \in X,
\end{equation}
where
\begin{equation*}
H_\bs(\by)=\bigotimes_{j \in \NN}H_{s_j}(y_j),\quad v_\bs:=\int_{\RRi} v(\by)\,H_\bs(\by)\, \rd\gamma (\by), \quad \bs \in \FF,
\label{hermite}
\end{equation*}
with $(H_k)_{k\in \NN_0}$ being the  Hermite polynomials normalized according to
$\int_{\RR} | H_k(y)|^2\, g(y)\, \rd y= 1.$  Here recall that 
$\FF$  denotes the set of all sequences of non-negative integers $\bs=(s_j)_{j \in \NN}$ such that their support $\nu_\bs:=\{j \in \NN: s_j >0\}$ is a finite set.
Notice that $(H_\bs)_{\bs \in \FF}$ is an orthonormal basis of $L_2(\RRi,\gamma):= L_2(\RRi,\RR, \gamma)$. 
 Moreover, for every $v \in \Ll_2(X)$  represented by the 
series \eqref{series},  the Parseval's identity holds
\begin{equation} \nonumber
\|v\|_{\Ll_2(X)}^2
\ = \ \sum_{\bs\in\FF} \|v_\bs\|_X^2.
\end{equation}

For $\bs,\bs'\in \FF$, the inequality $\bs'\leq \bs$ means that $s'_j\leq s_j$ for $j\in \NN$.
A set $\Lambda\subset \FF$ is called downward closed if the inclusion $\bs\in \FF$ yields the inclusion $\bs'\in \FF$ for every  $\bs'\in \FF$
such that $\bs'\leq \bs$. A sequence $(\sigma_\bs)_{\bs\in \FF}$ is called increasing if
$\sigma_{\bs'}\leq \sigma_{\bs}$ when $\bs'\leq \bs$. 
Denote by  $(\be^j)_{j\in \NN}$  the standard basis of $\ell_2(\NN)$.

In this subsection, as the first preliminary step in deep ReLU neural network approximation, we study the approximation of $v \in \Ll_2(X)$ by truncations $S_\Lambda$ on a finite set $\Lambda \subset \FF$ of the Hermite gpc expansion \eqref{series}. In the second preliminary step, we study the approximation of $v$ by the truncations of  $S_\Lambda$ on finite-dimensional supercube. In these approximations, functions $v$ are assumed to satisfy a certain weighted $\ell_2$-summability condition as described in Assumption A$_{q,\theta}$ below.

For $\theta \ge 0$, we define the sequence
\begin{equation} \label{[p_s]}
	p_\bs(\theta) := \prod_{j \in \NN}(1 +  s_j)^\theta, \quad \bs \in \FF.
\end{equation} 

For $0 < q < \infty$ and $\theta \ge 0$, we say that $v \in \Ll_2(X)$  represented by the series \eqref{series}, satisfies Assumption A$_{q,\theta}$ (A$_{q}$ $=$ A$_{q,0}$ for short) if

\medskip
\noindent
{\bf Assumption A$_{q,\theta}$} \ 	 There exist a constant $M$ and an increasing 
sequence  $\bsigma =(\sigma_\bs)_{\bs \in \FF}$ of positive 
numbers such that   $\sigma_{\be^{i'}} \le \sigma_{\be^i}$ if $i' 
< i$ , and that
\begin{equation} \label{ell_2-summability}
\left(\sum_{\bs\in\FF} (\sigma_\bs \|v_\bs\|_{X})^2\right)^{1/2} \ \le M \ <\infty, \ \ \text{with} \ \  \big(p_\bs(\theta)\sigma_\bs^{-1}\big)_{\bs\in \FF} \in {\ell_{q}}(\FF).
\end{equation}

We would like to emphasize that the weighted $\ell_2$-summability condition \eqref{ell_2-summability} involving the function $p_\bs(\theta)$ with $\theta > 0$ play a vital role in the proofs of the main results of the present paper. For example, the proof of Theorem \ref{thm1-dnn} requires Assumption A$_{q,\theta}$ with $\theta \ge 4/q$.
Notice that if  $v$ satisfies Assumption A$_{q,\theta}$, then $v$ also  satisfies Assumption A$_{q,\theta'}$ for every $\theta' < \theta$. In what follows, we will use this property without mention.

Assume that $0 < q < \infty$ and 
$\bsigma=(\sigma_{\bs})_{\bs \in \FF}$  is  an increasing sequence of positive numbers.
For $\xi>0$, we introduce the set 
\begin{equation} \label{Lambda(xi)}
\Lambda(\xi)
:= \ 
\big\{\bs \in \FF: \, \sigma_{\bs}^{q} \le \xi \big\}.
\end{equation}
For a function $v \in \Ll_2(X)$ represented by the series \eqref{series}, we define the truncation 
\begin{equation} \label{S_{Lambda(xi)}v}
	S_{\Lambda(\xi)} v 
	:= \ 
	\sum_{\bs\in {\Lambda(\xi)}} v_\bs H_\bs.
\end{equation}

\begin{lemma}\label{lemma[L_2-approx]} 
	For every $v \in \Ll_2(X)$ satisfying Assumption A$_q$ and
	for every $\xi >1$, there holds
	\begin{equation} \nonumber
		\|v- S_{\Lambda(\xi)}v\|_{\Ll_2(X)} \leq M \xi^{- 
			1/q}.
	\end{equation}
\end{lemma}

\begin{proof}	
	Applying the Parseval's identity, 
	noting~\eqref{S_{Lambda(xi)}v},~\eqref{Lambda(xi)} and 
	Assumption A$_q$, we obtain
	\begin{equation} \nonumber
		\begin{split}
			\|v- S_{\Lambda(\xi)}\|_{\Ll_2(X)}^2 
			\ &= \
			\sum_{\sigma_{\bs}> \xi^{1/q} } \|v_\bs\|_{X}^2 
			\ = \
			\sum_{\sigma_{\bs}> \xi^{1/q} } (\sigma_{\bs}\|v_\bs\|_{X})^2 \sigma_{\bs}^{-2}
			\\[1.5ex]
			\ &\le \
			\xi^{-2/q}
			\sum_{\bs \in \FF} (\sigma_{\bs}\|v_\bs\|_{X})^2 
			\ = \ M^2 \xi^{-2/q}.
		\end{split}
	\end{equation} 
	\hfill
\end{proof}

The set $\Lambda(\xi)$ determining the truncation $S_{\Lambda(\xi)}$ plays an important role in deep ReLU neural network approximation in the Bochner space $\Ll_2(X)$. Let us present some properties of this set which will be used in the following. Observe that under Assumption A$_q$, the set $\Lambda(\xi)$ is finite and downward closed.
We define the following numbers:	
\begin{equation} \label{m_p(xi)}
m_1(\xi)
:= \ 
\max_{\bs \in \Lambda(\xi)} |\bs|_1,
\end{equation}	
and	
\begin{equation} \label{m(xi)}
	m(\xi)
	:= \ 
	\max\big\{j \in \NN:  \exists \bs \in \Lambda(\xi)\ {\rm such \ that} \ s_j > 0 \big\}.
\end{equation}
By this definition we have
	\begin{equation} \label{bigcup}
		\bigcup_{\bs \in \Lambda(\xi)}\nu_{\bs} \subset  \{1,2,\ldots, m(\xi)\},
	\end{equation}
where $\nu_{\bs}:= \supp (\bs):= \{j \in \NN: s_j >0\}$ denotes the support of $\bs$.

We put for $\big(\sigma_{\bs}^{-1}\big)_{\bs\in	\FF}\in \ell_q(\FF)$ 
\begin{align} \label{K_q}
	K_q:= 	\max\brac{1, K_q'} \ \ \text{with}  \ \
	K_q':= \sum_{\bs\in \FF} \sigma_{\bs}^{-q}, 
\end{align}
and for $\big(p_{\bs}(\theta)\sigma_{\bs}^{-1}\big)_{\bs\in 
	\FF}\in \ell_q(\FF)$
\begin{align} \label{K_{q,theta}}
	K_{q,\theta}:=	\max\brac{1, 	K_{q,\theta}'} \ \ \text{with}  \ \
	K_{q,\theta}':= \left(\sum_{\bs\in \FF} p_{\bs}(\theta)^q \sigma_{\bs}^{-q}\right)^{\frac{1}{ \theta q}} . 
\end{align}

\begin{lemma}\label{lemma: |s|_1, |s|_0}
	Let  $\xi \ge 1$ and $(\sigma_{\bs})_{\bs \in \FF}$  be a sequence of positive numbers. Then we have the following.
	\begin{itemize}
		\item[{\rm (i)}]	
		Assume that  $\big(\sigma_{\bs}^{-1}\big)_{\bs\in	\FF}\in \ell_q(\FF)$. The set $\Lambda(\xi)$ is finite and it holds
		\begin{align*}
			|\Lambda(\xi)|
			\le K_q \xi.
		\end{align*}
		\item[{\rm (ii)}]
		Assume that  $\big(p_{\bs}(\theta)\sigma_{\bs}^{-1}\big)_{\bs\in 
			\FF}\in \ell_q(\FF)$ for some $\theta > 0$. There holds
		\begin{align*} 
			m_1(\xi)
			\le K_{q,\theta} \xi^{\frac{1}{\theta q}}.
		\end{align*}
	\end{itemize}
\end{lemma}

\begin{proof}
	Notice that
	$1 \le \sigma_{\bs}^{-q} \xi$  for every $\bs \in \Lambda(\xi)$.
	This 
	implies (i):
	\begin{align}
		|\Lambda(\xi)|
		\ = \		
		\sum_{\bs\in \Lambda(\xi)} 1
		\le 
		\sum_{\bs\in \Lambda(\xi)}	\xi \sigma_{\bs}^{-q}
		\le K_q \xi
		\notag
	\end{align}
	Moreover, we have that $1 \le s_j $ for every $j \in \nu_\bs$. Hence, we derive the inequalities
	\begin{align}
		\max_{\boldsymbol{s}\in \Lambda(\xi)} |\bs|_1^{\theta q} 
		\le 
		\sum_{\bs\in \Lambda(\xi)}
		\brac{\prod_{j \in \nu_\bs} 
			(1 + s_j)}^{\theta q} \leq  \sum_{\bs\in \Lambda(\xi)} 	
		p_\bs(\theta)^{q}\xi \sigma_{\bs}^{-q}
		\le K_{q,\theta}^{\theta q} \xi
		\notag
	\end{align}
	which  prove (ii).
	\hfill
\end{proof}

\begin{lemma}\label{lemma: m(xi)}
	Let $0 < q < \infty$ and 
	$(\sigma_{\bs})_{\bs \in \FF}$  be an increasing 
	sequence of positive numbers. 
	Assume that  $\big(\sigma_{\bs}^{-1}\big)_{\bs\in	
		\FF}\in \ell_q(\FF)$ and $\sigma_{\be^{i'}} \le 
	\sigma_{\be^i}$ if $i' < i$ . Then there holds for $\xi \ge 1$
	\begin{align} \label{m(xi)<}
		m(\xi)
		\le K_q \xi. 
	\end{align} 
\end{lemma}

\begin{proof}The statement is obvious if  $\Lambda(\xi)=\emptyset$. Therefore we assume that $\Lambda(\xi)\not=\emptyset$. Noting~\eqref{m(xi)}, there is a $\bs  
	\in \Lambda(\xi)$ 
	such that $s_{m(\xi)} \geq 1$. Then we have  
	$\be^{m(\xi)}  \le \bs$. Here recall that $\be^{m(\xi)}$ is the $m(\xi)$th vector in the standard basis of $\ell_2(\NN)$. Since $(\sigma_{\bs})_{\bs \in \FF}$  be  an increasing 
		sequence, we have $\Lambda(\xi)$ is 
	downward closed and therefore   $\be^{m(\xi)}  \in 
	\Lambda(\xi)$. From the definition~\eqref{Lambda(xi)} of 
	$\Lambda(\xi)$ and 
	the assumption in the lemma, we obtain
	\[
	\sigma_{\be^1}^q
	\le
	\sigma_{\be^2}^q
	\le 
	\ldots 
	\le 
	\sigma_{\be^{m(\xi)}}^q
	\le 
	\xi.
	\]
	Thus,
	$\be^1,\ldots, \be^{m(\xi)}$ belong to $\Lambda(\xi)$. 
	This yields the inequality $|\Lambda(\xi)| \ge {m(\xi)}$
	which together with the inequality $|\Lambda(\xi)| 
	\le K_q \xi$ in Lemma \ref{lemma: |s|_1, |s|_0}(i) proves 
	\eqref{m(xi)<}.
	\hfill 
\end{proof}

 Notice that if Assumption A$_q$ holds, then $m(\xi)$ is finite and, therefore, the truncation $S_{\Lambda(\xi)} v $ of the series \eqref{series} can be seen as a function on $\RR^{m(\xi)}$.  In this section, for $\xi >1$, we use the letters 
	$\omega$ and $m$ only for the notation
	\begin{align} \label{m,omega}
	\omega := \lfloor {K_{q,\theta}}\xi \rfloor, \ \ m:= m(\xi). 
	\end{align}

Let us  introduce the truncation $S_{\Lambda(\xi)}^{\omega}v$ of $S_{\Lambda(\xi)}v$ on the supercube 
	$$
	B^m_\omega
:=
[-2\sqrt{\omega}, 2\sqrt{\omega}]^m \subset \RR^m.	
	$$
	 
For a function $\varphi$ defined on $\RR$, 
we denote by $\varphi^{\omega}$ the truncation of $\varphi$ on $B^1_\omega$, i.e.,
\begin{align*}
	\varphi^{\omega}(y)
	:=
	\begin{cases}
		\varphi(y)
		&
		\text{if }
		y \in B^1_\omega
		\\
		0
		&
		\text{otherwise}.
		\label{truncation}
	\end{cases}
\end{align*}
If $\nu_{\bs} \subset \{1,\ldots,m\}$, we 
put
$$
H_\bs^{\omega}(\boldsymbol{y}) := \prod_{j=1}^m H_{s_j}^{\omega}(y_j),\qquad \by\in \RR^m. 
$$
We have
$
H_\bs^{\omega}(\boldsymbol{y})
=
\prod_{j=1}^m H_{s_j}(y_j)
$
if 
$
\boldsymbol{y}\in B^m_\omega $ and $H_\bs^{\omega}(\boldsymbol{y})=0$ otherwise.

For a function $v \in \Ll_2(X)$ represented by the 
series \eqref{series}, noting the truncation 
$S_{\Lambda(\xi)} v $ given by \eqref{S_{Lambda(xi)}v} and 
\eqref{Lambda(xi)},
we define
\begin{equation} \label{S_{Lambda(xi)}^omega}
	S_{\Lambda(\xi)}^\omega v 
	:= \ 
	\sum_{\bs\in {\Lambda(\xi)}} v_\bs H_\bs^\omega.
\end{equation}
Below in this subsection, we use letters $C$ and $K$ to denote various constants which may depend on 
the parameters $M,q,\bsigma,\theta,$ as mentioned in Theorem \ref{thm1-dnn}.

In what follows, for convenience we consider $\RR^m$ as  the subset of all $\by \in \RRi$ such that $y_j = 0$ for $j = m+1,\ldots$.  If $f$ is a function on $\RRm$ taking values in a Hilbert space $X$, then $f$ has an extension to $\RR^{m'}$ for $m' > m$ or to $\RRi$ which is denoted again by $f$, by the formula 
	$f(\by)= f \brac{(y_j)_{j=0}^m}$ for $\by = (y_j)_{j =1}^{m'}$ or $\by = (y_j)_{j \in \NN}$, respectively. 
	
	The tensor product of standard Gaussian probability measures $\gamma(\by)$ on  $\RR^m$ is defined by
	\begin{equation} \nonumber
		\rd \gamma(\by) 
		:= \ 
		\bigotimes_{j = 1}^m g(y_j) \rd (y_j).
	\end{equation}
For a $\gamma$-measurable subset $\Omega$ in $\RRm$, the spaces  $L_2(\Omega,X,\gamma)$ and $L_2(\Omega,\gamma)$ are defined in the usual way. 

Under the assumptions of Lemma 
	\ref{lemma: m(xi)},  $S_{\Lambda(\xi)} v $  can be 
	seen as a function on $\RRm$, where we recall that $m := m(\xi)$. The function $S_{\Lambda(\xi)}^{\omega}v$ can be treated as the truncation of 
	$S_{\Lambda(\xi)} v $  on the supercube
{$ B^m_\omega $}.
For $g \in L_2(\RRm,X,\gamma)$, we have $\norm{g}{ L_2(\RRm,X,\gamma)} = \norm{g}{ L_2(\RRi,X,\gamma)}$ in the sense of extension of $g$. We will make use of these facts  without mention.
 
We are interested in the estimating in terms of parameter $\xi$ the error of the  approximation of $v \in \Ll_2(X)$ by $S_{\Lambda(\xi)}^{\omega}v$.  To this end, we need some bound for the  $L_2(\RR^m{\setminus}B^m_\omega,\gamma)$ norm of polynomials on $\RRm$.
\begin{lemma}\label{l:Rm trun1}
	Let $\varphi (\by)= \prod_{j = 1}^m \varphi_j(y_j)$ 
	for  $\by \in \RR^m$, where $\varphi_j$ 
	is a polynomial in the variable $y_j$
	of degree not greater than $\omega$ for 
	$j=1,\ldots,m$. 
	Then there holds 
	\begin{align*} 
		\norm{\varphi}{L_2(\RR^m{\setminus}B^m_\omega,\gamma)}
		\le 
		Cm
		\exp \left(- K\omega \right)
		\norm{\varphi}{L_2(\RR^m,\gamma)},
	\end{align*}
	where the positive constants $C$ and $K$ are independent of $\omega$, $m$ and $\varphi$.
\end{lemma}

\begin{proof}
	The proof of the lemma relies on the following inequality which is an immediate consequence of  \cite[Theorem 6.3]{Lu07}.
	Let $\psi$ be a polynomial of degree 
	at most $\ell$.  Applying \cite[Theorem 6.3]{Lu07} for polynomial ${\psi(\sqrt{2}t)}$ with weight $\exp(-t^2)$ (in this case $a_{\ell}=\sqrt{\ell}$, see \cite[Page 41]{Lu07}) and $\kappa=\sqrt{2}-1$ we get
	\begin{equation} \nonumber
		\|\psi \|_{L_2( \RR \setminus [-2\sqrt{{\ell}},2\sqrt{{\ell}}\,], \gamma)} 
		\ \le \
		C\, 	\exp(- K {\ell}) \|\psi \|_{L_2([-\sqrt{{2\ell}},{\sqrt{{2\ell}}}\,], \gamma)}
	\end{equation}
	for some positive numbers $C$ and $K$ independent of $\ell$ and $\psi$. Hence, for a polynomial $\psi$  of degree not greater than $\omega$ we obtain
	\begin{equation} \label{[norm-ineq]}
		\norm{\psi}{L_2(\RR\setminus B^1_{\omega},\gamma)}
			\ \le \
		C\, 	\exp(- K {\omega}) \|\psi \|_{L_2([-\sqrt{{2\omega}},{\sqrt{{2\omega}}}\,], \gamma)}
		\ \le \
		C \exp \left(- K\omega \right) \norm{\psi}{L_2(\RR,\gamma)}.
	\end{equation}
	
	We denote 
	\begin{equation*} \label{I_j}
		I_j
		:=
		\big\{\by =(y_i)_{i=1}^m\in \RR^m: |y_j|> 2\sqrt{\omega}\big\}, \ \ j = 1,..., m.
	\end{equation*}	
		Since 
		$$
		\RR^m \setminus B^m_\omega = \bigcup_{j=1}^m I_j,
		$$
 we have
	\begin{align}
		\norm{\varphi}{L_2(\RR^m{\setminus}B^m_\omega,\gamma)}
		\ \le \
		\sum_{j=1}^m \norm{\varphi}{L_2(I_j,\gamma)}
		\ = \ 
		\sum_{j=1}^m \bigg( \norm{\varphi_j}{L_2(\RR\setminus B^1_{\omega},\gamma)}\prod_{i\not=j}\norm{\varphi_i}{
			L_2(\RR ,\gamma)}\bigg).
		\label{RmIj1} 
	\end{align}
	Applying~\eqref{[norm-ineq]} for the polynomials $\varphi_j$,  for $j=1,\ldots,m$, whose degree is not greater than $\omega$ we obtain
	\begin{align}
		\norm{\varphi_j}{L_2(\RR\setminus B^1_{\omega},\gamma)}
		\ &\le \
		C \exp \left(- K\omega \right) \norm{\varphi_j}{L_2(\RR,\gamma)}
		\notag
	\end{align}
	with some positive constants $C$ and $K$ independent of $\omega$, $m$ and $\varphi$.
	This together with~\eqref{RmIj1} yields
	\begin{align}
		\norm{\varphi}{L_2(\RR^m{\setminus}B^m_\omega,\gamma)}
		\ \le \
		C \exp \left(- K\omega \right) m \prod_{i=1}^m\norm{\varphi_i}{L_2(\RR,\gamma)}
		\ = \
		C m \exp \left(- K\omega \right) \norm{\varphi}{L_2(\RR^m,\gamma)}.
		\notag
	\end{align}
	\hfill
\end{proof}

\begin{lemma}\label{lemma:v-S^omega} 
	Let $0 < q  < \infty$ and $\theta\geq 1/q$.
	For every $v \in \Ll_2(X)$ satisfying Assumption A$_{q,\theta}$ and
for every $\xi > 1$, there holds
\begin{equation} \label{S-S^omega}
	\|v- S_{\Lambda(\xi)}^{\omega} v\|_{\Ll_2(X)} \leq C \xi^{- 1/q},
\end{equation}
	where the positive constant $C$  is independent of $v$ and $\xi$.
	
\end{lemma}

\begin{proof} Since $S_{\Lambda(\xi)}v $ and $S_{\Lambda(\xi)}^{\omega} v$ can be considered as functions on $\RRm$, we have
		\begin{equation}\nonumber 
		\|S_{\Lambda(\xi)}v - S_{\Lambda(\xi)}^{\omega}v\|_{\Ll_2(X)}
		= \|S_{\Lambda(\xi)}v - S_{\Lambda(\xi)}^{\omega}v\|_{{L_2(\RRm \setminus B^m_\omega,X,\gamma)}},
	\end{equation}
and consequently,
	\begin{equation}\label{}
	\|v- S_{\Lambda(\xi)}^{\omega} v\|_{\Ll_2(X)} 
	\leq \|v-S_{\Lambda(\xi)}v  \|_{\Ll_2(X)}
	+ \|S_{\Lambda(\xi)}v - S_{\Lambda(\xi)}^{\omega}v\|_{{L_2(\RRm \setminus B^m_\omega,X,\gamma)}}.
\end{equation}
From Lemma \ref{lemma[L_2-approx]}, it is sufficient to show  that the second term in the right-hand side is bounded by $C \xi^{- 1/q}$. 
	By  the equality 
	$$
	\ \norm{H_{\bs} - H_{\bs}^{\omega}}{L_2(\RR^m,\gamma)}
	\ = \ \norm{H_{\bs}}{L_2(\RR^m\setminus B^m_\omega,\gamma)},
	\quad \bs\in \Lambda(\xi),
	$$
	and the triangle inequality  we obtain 
	\begin{align}
		\|S_{\Lambda(\xi)}v - S_{\Lambda(\xi)}^{\omega}v\|_{{L_2(\RRm \setminus B^m_\omega,X,\gamma)}}
		& 
		\leq 
		\sum_{\bs\in \Lambda(\xi)} 
		\norm{v_{\bs}}{X}
		\norm{H_{\bs} - H_{\bs}^{\omega}}{L_2(\RR^m,\gamma)}
		\notag
		\\
		&
		=
		\sum_{\bs\in \Lambda(\xi)} 
		\norm{v_{\bs}}{X}
		\norm{H_{\bs}}{L_2(\RR^m\setminus B^m_\omega,\gamma)}.
		\label{|S - S^omega|}
	\end{align}
	Notice that for every $\bs\in \Lambda(\xi)$, 
	$$
	H_\bs(\by)= \prod_{j = 1}^m H_{s_j}(y_j), \ \ \by \in \RR^m,
	$$
	 where $H_{s_j}$ is a polynomial in variable $y_j$, 
	of degree not greater than $m_1(\xi)$.  By 
		Lemma~\ref{lemma: |s|_1, |s|_0}(ii),  \eqref{m,omega} and the assumption $\theta\geq 1/q$ we have that
			 $$
			 m_1(\xi)\leq \lfloor K_{q,\theta} \xi^{\frac{1}{\theta q}}\rfloor\leq\lfloor K_{q,\theta} \xi \rfloor=  \omega.
			 $$ 
	Applying Lemma~\ref{l:Rm trun1} gives
	\begin{align}
		\norm{H_\bs}{L_2(\RR^m{\setminus}B^m_\omega,\gamma)}
		\le 
		Cm
		\exp \left(- K\omega \right)
		\norm{H_\bs}{L_2(\RR^m,\gamma)}
		= 
		Cm
		\exp \left(- K\omega \right),
		\notag
	\end{align}
	where the positive constants $C$ and $K$ are independent of $\omega$, $m$ and $\bs$.
	This together with \eqref{|S - S^omega|} and the Cauchy--Schwarz  inequality yields that
	\begin{equation}\label{S_Lambda - S_Lambda^omega}
		\begin{aligned}  
			\|S_{\Lambda(\xi)}v - S_{\Lambda(\xi)}^{\omega}v\|_{{L_2(\RRm \setminus B^m_\omega,X,\gamma)}}
			&
			\leq 
			C m
			\exp 
			\brac{- K\omega} 
			\sum_{\bs\in \Lambda(\xi)} 
			\norm{v_{\bs}}{X}
			\\
			&
			\leq 
			C m 
			\exp \brac{- K\omega} \sup_{\bs\in\FF}  \sigma_\bs^{-1}   \sum_{\bs\in \Lambda(\xi)}\sigma_{\bs} \norm{v_{\bs}}{X}
			\\
			&
			\leq 
			C m 
			\exp \brac{- K\omega} {|\Lambda(\xi)|^{1/2}} \bigg( \sum_{\bs\in \Lambda(\xi)}(\sigma_{\bs} \norm{v_{\bs}}{X})^2\bigg)^{1/2} .
		\end{aligned}
	\end{equation}
Using  \eqref{m,omega}, Assumption A$_{q,\theta}$ and Lemmata \ref{lemma: |s|_1, |s|_0} and \ref{lemma: m(xi)} we finally obtain
	\begin{equation} 
\begin{aligned}  
\|S_{\Lambda(\xi)}v - S_{\Lambda(\xi)}^{\omega}v\|_{{L_2(\RRm \setminus B^m_\omega,X,\gamma)}}
\leq 
C  
\xi^{3/2} \exp \brac{- K \xi}
\le 
C\xi^{-1/q}.
\end{aligned}
\end{equation}
This is the bound of the second term and therefore the stated result is proved.
	\hfill
\end{proof}

\subsection{Approximation by deep ReLU neural networks}

In this subsection, we construct deep ReLU neural networks which can be used to approximate functions in $\Ll_2(X)$, and prove  bounds of the error of  the approximation by them. 
%

The main result in this subsection is read as follows.

\begin{theorem}\label{thm1-dnn}
	Let  $v \in \Ll_2(X)$ satisfy Assumption A$_{q,\theta}$ with
	$\theta \ge 4/q$. 
Then for every $\xi > 2$, we can  construct  a deep 
ReLU neural network $\bphi_{\Lambda(\xi)}:= (\phi_\bs)_{\bs 
\in \Lambda(\xi)}$ on 
$\RR^m$,  $m\leq \lfloor K_q \xi\rfloor$, having the following properties. 
\begin{itemize}
	\item [{\rm (i)}]
The deep ReLU neural network $\bphi_{\Lambda(\xi)}$ is independent of $v$;
	
\item [{\rm (ii)}] 
The input and output dimensions of $\bphi_{\Lambda(\xi)}$ are at most $m$;

\item [{\rm (iii)}]
The  components $\phi_{\bs}$, $\bs\in \Lambda(\xi)$, of $\bphi_{\Lambda(\xi)}$ are deep ReLU neural networks on 
$\RR^{|\nu_\bs|}$ with ${|\nu_\bs|}  \le K_{q,\theta} \xi^{\frac{1}{\theta q}}$, having support contained in the super-cube $[-T,T]^{|\nu_\bs|}$, 
where  $T:= 4\sqrt{\lfloor K_{q,\theta} \xi \rfloor}$;
	
\item [{\rm (iv)}] 
$W\big(\bphi_{\Lambda(\xi)}\big) \le C\xi \log \xi$;

\item [{\rm (v)}] 
$L\big(\bphi_{\Lambda(\xi)}\big) \le C \xi^{1/\theta q}$;

\item [{\rm (vi)}] 	If $\phi_\bs$ is extended to the whole $\RRi$ by $\phi_\bs(\by) =\phi_\bs\brac{\big(y_j\big)_{j\in \nu_\bs}}$ for 
$\by = \big(y_j\big)_{j\in \NN} \in \RRi$, the approximation of $v$ by  the function
\begin{equation} \label{Phi_v}
	\Phi_{\Lambda(\xi)}v (\by) := \sum_{\bs \in \Lambda(\xi)} v_\bs \phi_\bs(\by)
\end{equation}
gives the error estimate 
\begin{align} \label{approximation-error}
\| v- \Phi_{\Lambda(\xi)}v  \|_{\Ll_2(X)}\leq C\xi^{-1/q}.
\end{align}
\end{itemize}
Here the positive constants $C = C_{M,q,\bsigma,\theta}$ are independent of $v$ and $\xi$.
\end{theorem}

\begin{proof}  
For convenience we brief the proof into several steps. Recall that throughout the present paper, we use letters $C$  and $K$ to denote general positive constants which may take different values.
\\
\\ 
{\bf Step 1.} In this step, we construct a deep ReLU neural network  $\bphi_{\Lambda(\xi)}:= (\phi_\bs)_{\bs \in \Lambda(\xi)}$ on $\RRm$ and prove the claims (i)--(iii) in Theorem \ref{thm1-dnn}. 

We preliminarly notice the following. Suppose that $\bphi_{\Lambda(\xi)}$ and therefore, the function $\Phi_{\Lambda(\xi)}$ are already constructed. Due to the inequality
\begin{equation}\label{eq-decompose}
\begin{aligned}
\|v- \Phi_{\Lambda(\xi)} v\|_{\Ll_2(X)}
& \leq \|v-S_{\Lambda(\xi)}^{\omega}v  \|_{\Ll_2(X)}
\\
& 
+ \| S_{\Lambda(\xi)}^{\omega}v- \Phi_{\Lambda(\xi)} v\|_{L_2(B^m_\omega,X,\gamma)}
+ \| \Phi_{\Lambda(\xi)} v  \|_{L_2(\RRm \setminus B^m_\omega,X,\gamma)},
\end{aligned}
\end{equation}
the claim (vi) will be proven if we show the bound  $C\xi^{-1/q}$ for the three terms in the right-hand side. This bound has been shown for  first term   as in Lemma \ref{lemma:v-S^omega}. This tell us that we should construct 
$\bphi_{\Lambda(\xi)}$ on $\RRm$ for approximating $S_{\Lambda(\xi)}^{\omega}v$ by $\Phi_{\Lambda(\xi)}v$, and prove that the second and  third terms in \eqref{eq-decompose} are bounded by $C\xi^{-1/q}$. 

Let us begin to construct $\bphi_{\Lambda(\xi)}$.  Since the function $S_{\Lambda(\xi)}^{\omega}v$ is a linear combination of the truncated Hermite polynomials $H_\bs^\omega$, $\bs \in \Lambda(\xi)$,  on the supercube $B^m_\omega \subset \RRm$, we will use the construction in Lemma \ref{lem-product-varphi} to design a deep ReLU neural network $\phi_\bs$ for approximating each $H_\bs^\omega$ with an appropriate accuracy.  Then the deep ReLU neural network $\bphi_{\Lambda(\xi)}$ formed by parallelization as described in Lemma~\ref{lem:parallel} will be desired.

It is well-known  that for each $s\in \NN_0$, the univariate Hermite polynomial $H_s$
can be written as 
\begin{align}\label{eq-Hs}
	H_s(x)
	&
	=
\sqrt{s!} \sum_{\ell=0}^{\left\lfloor \tfrac{s}{2} 
	\right\rfloor} \frac{(-1)^\ell}{\ell!(s - 2\ell)!} \frac{x^{s - 
		2\ell}}{2^\ell}:=	\sum_{\ell=0}^s a_{s,\ell} x^\ell.
\end{align} 
Then for each  $\bs\in \Lambda(\xi)$  we have 
\begin{align} \label{eq-Hbs}
	H_{\bs}(\by)
	=
	\prod_{j=1}^m 
	H_{s_j}(y_j)
	=
	\sum_{\bell=\boldsymbol{0}}^{\bs} 
	\brac{
		\prod_{j=1}^m 
		a_{s_j,\ell_j}}
	\by^{\bell} = \sum_{\bell=\boldsymbol{0}}^{\bs}  a_{\bell}
	\by^{\bell},
\end{align}
where we put 
$a_{\bell}:= \prod_{j=1}^m a_{s_j,\ell_j}$ and 
$\by^{\bell}:= \prod_{j=1}^m y_j^{\ell_j}$.

From \eqref{S_{Lambda(xi)}^omega} and  \eqref{eq-Hbs} we get for every $\by \in B^m_\omega$,
\begin{align}
	S_{\Lambda(\xi)}^{\omega}v(\by)
	=
	\sum_{\bs\in \Lambda(\xi)} 
	v_{\bs} 
	\sum_{\bell=\boldsymbol{0}}^{\bs} 
		a_{\bell}
		\brac{\by^{\bell}}^{\omega}
=
\sum_{\bs\in \Lambda(\xi)} 
v_{\bs}
\sum_{\bell=\boldsymbol{0}}^{\bs} 
a_{\bell}
\brac{2\sqrt{\omega}}^{|\bell|_1}
\prod_{j \in \nu_\bell}\brac{\frac{y_j}{2\sqrt{\omega}}}^{\ell_j}.
\label{Slambda3}
\end{align}
We will use this form of representation of  $S_{\Lambda(\xi)}^{\omega}v$ for constructing  $\bphi_{\Lambda(\xi)}$ and $\Phi_{\Lambda(\xi)}$.

Let $\bell \in \FF$ be such that  
$\boldsymbol{0} \le \bell  \le \bs$. By definition we have $\nu_\bell \subset \nu_\bs $. 
By changing variables 
$$
\bx= \frac{\by}{2\sqrt{\omega}}, \ \ \bx \in \nu_\bs,
$$ 
we have
\begin{equation} 
	\prod_{j \in {\nu_\bell}} \left(\frac{y_j}{2\sqrt{\omega}}\right)^{\ell_j}
	= 
	\prod_{j \in \nu_{\bell}} 
	\varphi_1^{\ell_j}\brac{\frac{y_j}{2\sqrt{\omega}}} \prod_{j \in \nu_\bs \setminus \nu_{\bell}} 
	\varphi_0\brac{\frac{y_j}{2\sqrt{\omega}}}=h_{\bs,\bell}(\bx), 
	\ \ \by \in B^{|\nu_\bs|}_\omega,
\end{equation}
where
\begin{equation} \label{varphi=}
	h_{\bs,\bell}(\bx) 
	:=
	\prod_{j \in \nu_\bell} \varphi_1^{\ell_j}(x_j)
	\prod_{j \in \nu_\bs\setminus\nu_\bell}
	\varphi_0\brac{x_j}, 
\end{equation}
$\varphi_0$ and $\varphi_1$ are the piece-wise linear functions defined before Lemma~\ref{lem-product-varphi}.
Hence by Lemma \ref{lem-product-varphi}, for every  $\boldsymbol{0} \le \bell  \le \bs$,
with 
\begin{align}\label{eq-Absbell}
\delta_\bs^{-1}
:= 
\xi^{1/q+1/2} p_\bs(1)
\brac{2\sqrt{\omega}}^{|\bs|_1}
\displaystyle
\max_{\boldsymbol{0}\leq \bell\le \bs}\{|a_{\bell}|
\},
\end{align}
there exists a deep  ReLU neural network
$\phi_{\bs,\bell}$ on $\RR^{|\nu_\bs|}$
 with 
$\supp \big(\phi_{\bs,\bell}\big) \subset [-2,2]^{|\nu_\bs|}$ such that 
\begin{align}\label{k-01} 	
\sup_{\by \in B^{|\nu_\bs|}_{4\omega}}\abs{
	h_{\bs,\bell}\brac{ \frac{\by}{2\sqrt{\omega}}} -	\phi_{\bs ,\bell}\brac{ \frac{\by}{2\sqrt{\omega}}}}
\ \le \
\delta_\bs,
\end{align}
and therefore,
\begin{align}\label{k-011} 	
\sup_{\by \in B^{|\nu_\bs|}_{\omega}}\abs{\prod_{j \in \nu_\bs}  
	\brac{\frac{y_j}{2\sqrt{\omega}}}^{\ell_j} 	-	\phi_{\bs,\bell}\brac{ \frac{\by}{2\sqrt{\omega}}}}
	\ \le \
	\delta_\bs,
\end{align}
and
\begin{align}\label{supp(g^ell_s)} 
\supp\bigg(\phi_{\bs,\bell}\bigg(\frac{\cdot}{2\sqrt{\omega}}\bigg)\bigg) \subset B^{|\nu_\bs|}_{4\omega}.
\end{align}

Considering   the right-hand side of  \eqref{eq-Absbell}, we have $\omega \ge \xi  >1$ and therefore,
	\begin{align}\label{eq-Absbell-2}
		\brac{2\sqrt{\omega}}^{|\bs|_1}
		\displaystyle
		\max_{\boldsymbol{0}\leq \bell\le \bs}\{|a_{\bell}|\}
		\ge 
		2^{|\bs|_1} |a_{\bs}|
		\ge 1.
	\end{align}
The last inequality can be proven by using  \eqref{eq-Hs}, \eqref{eq-Hbs} and Stirling's approximation.
With the definition \eqref{eq-Absbell}, this yields that  
$$
|\bell|_1 + |\nu_\bs \setminus \nu_\bell| \leq |\bs|_1 \leq  p_\bs(1) \leq\delta_\bs^{-1}. 
$$
Hence, the size and the depth of $\phi_{\bs,\bell}$ are 
bounded as 
\begin{align} \label{W(phi_s,ell)}
	W\brac{\phi_{\bs,\bell}}
	&
	\leq
	C
	\brac{ 1 + \abs{\bs}_1 \brac{ \log\abs{\bs}_1 + \log  \delta_\bs^{-1}}} 
\leq
C
\brac{1 + \abs{\bs}_1   \log  \delta_\bs^{-1}}
\end{align}
and
\begin{align} \label{L(phi_s,ell)}
L\brac{\phi_{\bs,\bell}}
\leq
C
\brac{ 1 + \log \abs{\bs}_1 \brac{\log\abs{\bs}_1 + \log \delta_\bs^{-1}}} 
\leq
C
\brac{ 1 + \log	\abs{\bs}_1   \log \delta_\bs^{-1}}.
\end{align}

For approximating $H_\bs^\omega$,
we define the deep ReLU neural network $\phi_\bs$ on  $\RR^{|\nu_\bs|}$ by
\begin{align} \label{g_Lambda^*}
\phi_\bs (\by)
	&
	:=
	\sum_{\boldsymbol{0}\leq \bell\le \bs}
	 a_{\bell}\brac{2\sqrt{\omega}}^{|\bell|_1}
	\phi_{\bs,\bell}
	\brac{ \frac{\by}{2\sqrt{\omega}}},
\quad \by \in \RR^{|\nu_\bs|},
\end{align}
which is a parallelization of the component deep ReLU neural networks 
$\phi_{\bs,\bell}\brac{\frac{\cdot}{2\sqrt{\omega}}}$. 

We define $\bphi_{\Lambda(\xi)}:= (\phi_\bs)_{\bs \in \Lambda(\xi)}$ as the deep ReLU neural network on $\RRm$ realized by parallelization $\phi_\bs$, $\bs \in \Lambda(\xi)$. 

In what follows, according to the above convention,  for $\bs \in \Lambda(\xi)$, in some places   we identify  the functions $\phi_{\bs,\bell}$  and $\phi_\bs$ on $\RR^{|\nu_\bs|}$  with their extentions  on $\RRm$ or on $\RRi$ due to the inclusions  $\nu_\bs \subset \{1,...,m\} \subset \NN$. 

The input dimension of $\bphi_{\Lambda(\xi)}$ is not greater than $m$ which is at most $\lfloor K_q \xi \rfloor$
by Lemma~\ref{lemma: m(xi)}.	
The output dimension of $\bphi_{\Lambda(\xi)}$ is the number $|\Lambda(\xi)|$ which is at most $\lfloor K_q \xi \rfloor$
by Lemma~\ref{lemma: |s|_1, |s|_0}(i). 
From \eqref{supp(g^ell_s)} it follows the inclusion
\begin{align}\label{supp(g)} 
\supp (\phi_\bs) \subset B^{|\nu_\bs|}_{4\omega}, \ \ \bs \in \Lambda(\xi),
\end{align}
and from Lemma~\ref{lemma: |s|_1, |s|_0}(ii) the inequality ${|\nu_\bs|}  \le K_{q,\theta} \xi^{\frac{1}{\theta q}}$. The claims (i)--(iii) are proven.
\\
\\
{\bf Step 2.} In this step, we  prove the claim (vi) in Theorem \ref{thm1-dnn}. 

Due to the inequality \eqref{eq-decompose} and Lemma \ref{lemma:v-S^omega} to  prove the claim (vi) it is sufficient to show the bounds	
\begin{align} \label{S_{Lambda(xi)}^{omega}v- Phi_{Lambda(xi)} v}
\| S_{\Lambda(\xi)}^{\omega}v- \Phi_{\Lambda(\xi)} v\|_{{L_2(B^m_\omega,X,\gamma)}}
	& \leq 
	C\xi^{-1/q},
	\end{align}
and
			\begin{align}	\label{norm-g_Lambda}
	\| \Phi_{\Lambda(\xi)} v  \|_{{L_2(\RRm \setminus B^m_\omega,X,\gamma)}}
	& \leq 
	C\xi^{-1/q},
\end{align}  
	where the positive constants $C$ are independent of $v$ and $\xi$.
	
From the equality
\begin{align} \nonumber
	\Phi_{\Lambda(\xi)}(\by)
	=
	\sum_{\bs\in \Lambda(\xi)} 
	v_{\bs}
	\sum_{\bell=\boldsymbol{0}}^{\bs} 
	a_{\bell}\phi_{\bs,\bell}
	\brac{ \frac{\by}{2\sqrt{\omega}}}
\end{align}
and \eqref{Slambda3}, \eqref{k-01}, \eqref{g_Lambda^*}, similarly to \eqref{S_Lambda - S_Lambda^omega},  we prove \eqref{S_{Lambda(xi)}^{omega}v- Phi_{Lambda(xi)} v}:
\begin{align}
\| S_{\Lambda(\xi)}^{\omega}v- \Phi_{\Lambda(\xi)} v\|_{{L_2(B^m_\omega,X,\gamma)}}
&=
\notag
\norm{\sum_{\bs\in \Lambda(\xi)} 
	v_{\bs} 
	H_{\bs}^{\omega}-  \sum_{\bs \in \Lambda(\xi)} v_\bs \phi_\bs(\by) }{L_2(B^m_\omega,X,\gamma)}
\\
& \leq 
\notag
\sum_{\bs\in \Lambda(\xi)} 
\norm{v_{\bs}}{X} 
\sum_{\bell=\boldsymbol{0}}^{\bs} 
|a_{\bell}|
\brac{2\sqrt{\omega}}^{|\bs|_1}\delta_\bs
\\
&
 \leq 
\xi^{-1/q-1/2}\sum_{\bs\in \Lambda(\xi)} 
\norm{v_{\bs}}{X} 
\le 
C\xi^{-1/q}.
\notag
\end{align}

We now verify \eqref{norm-g_Lambda}.
We first prove the following auxiliary inequality
\begin{align}\label{phi^{bs - be;bk}_{bell}<}
	\left|\phi_{\bs,\bell}
	\left( 
	\frac{\by}{2\sqrt{\omega}}\right)\right| \leq 2, \ \forall \by \in \RRm. 
\end{align}
Due to \eqref{supp(g^ell_s)}, it is sufficient to prove this inequality for 
$\by \in B^{|\nu_\bs|}_{4\omega}$. Observe that the inequalities \eqref{eq-Absbell-2} and $p_\bs(1) \ge 1$ yield $\delta_\bs  \le 1$. 	
On the other hand, from the definition of $h_{\bs,\bell}$ it follows
\begin{align}\nonumber
	\sup_{\by \in B^{|\nu_\bs|}_{4\omega}}	\left|h_{\bs,\bell}
	\left( 
	\frac{\by}{2\sqrt{\omega}}\right)\right| 
	\ \le \
	1
\end{align}
By using  the last two  inequalities, from \eqref{k-01}   we prove
\eqref{phi^{bs - be;bk}_{bell}<} for $\by \in B^{|\nu_\bs|}_{4\omega}$.

By  \eqref{g_Lambda^*}   and \eqref{phi^{bs - be;bk}_{bell}<} we have that
\begin{align*} 
\norm{\Phi_{\Lambda(\xi)} v}{{L_2(\RRm \setminus B^m_\omega,X,\gamma)}}
& 
\leq
\sum_{\bs\in \Lambda(\xi)}
\|	v_{\bs}\|_X
\sum_{\bell = \boldsymbol{0}}^{\bs}
\Big|
a_{\bell}\brac{2\sqrt{\omega}}^{|\bs|_1}\Big|
\norm{\phi_{\bs,\bell}
	\brac{ 
		\frac{\cdot}{2\sqrt{\omega}}}}{L_2(\RRm \setminus B^m_\omega,\gamma)}
\\
& 
\leq
2 \sum_{\bs\in \Lambda(\xi)}
\|	v_{\bs}\|_X
\sum_{\bell = \boldsymbol{0}}^{\bs}
\Big|
a_{\bell}\brac{2\sqrt{\omega}}^{|\bs|_1}\Big|
\norm{1}{L_2(\RRm \setminus B^m_\omega,\gamma)}.
\end{align*}
Applying Lemma \ref{l:Rm trun1} to the polynomial  $\varphi(\by) = 1$, we get	
\begin{align*} 
\norm{\Phi_{\Lambda(\xi)} v}{{L_2(\RRm \setminus B^m_\omega,X,\gamma)}}
& 
\leq 
C m 	\sum_{\bs\in \Lambda(\xi)}
\|	v_{\bs}\|_X \sum_{\bell = \boldsymbol{0}}^{\bs} (4\omega)^{|\bs|_1/2}  \exp(-K\omega) \big|  
a_{\bell} \big|
\\
&
= 
C m	\sum_{\bs\in \Lambda(\xi)}
\|	v_{\bs}\|_X
(4\omega)^{|\bs|_1/2} \exp(-K\omega) \sum_{\bell = \boldsymbol{0}}^{\bs} \big|a_{\bell} \big|. 
\end{align*}
In order to estimate the sum $\sum_{\bell = \boldsymbol{0}}^{\bs} \big|a_{\bell} \big|$, we need an inequality for the coefficients of Hermite polynomials. 
By the representation \eqref{eq-Hs} of 
$H_s$, $s\in \NN_0$, there holds
\begin{equation} \label{ineq-coeff}
\sum_{\ell=0}^{s} |a_{s,\ell}| \leq \sqrt{s!}.
\end{equation}
Indeed, this inequality  is obvious with $s=0,1,2,3$. When $s\geq 4$ we have $\frac{1}{\ell!(s-2\ell)!}\leq \frac{1}{2}$ for all $\ell=0,\ldots,\lfloor s/2\rfloor$. Therefore,
$$
	\sum_{\ell=0}^{s} |a_{s,\ell}| \leq \sqrt{s!} \sum_{\ell=0}^{\left\lfloor \tfrac{s}{2} 
		\right\rfloor} \frac{2^{-\ell} }  {\ell!(s - 2\ell)!}   
	\leq 
	\frac{\sqrt{s!}}  {2}  \sum_{\ell=0}^{\left\lfloor \tfrac{s}{2} 
		\right\rfloor}  2^{-\ell}  \leq \sqrt{s!}.
$$
It follows from \eqref{ineq-coeff} that
\begin{equation}\label{eq-sum-abell}
\sum_{\bell = \boldsymbol{0}}^{\bs} \big|  
a_{\bell} \big| = \sum_{\bell = \boldsymbol{0}}^{\bs} \prod_{j=1}^m 
\big|a_{s_j,\ell_j}\big| \leq \prod_{j=1}^m \sum_{\ell_j=0}^{s_j}|a_{s_j,\ell_j}| \leq \prod_{j=1}^m \sqrt{s_j!},   
\end{equation}
and hence,
\begin{equation}\label{sum-a_bell}
\sum_{\bell = \boldsymbol{0}}^{\bs} \big|  
a_{\bell} \big|   \leq \prod_{j=1}^m  \sqrt{s_j!} 
\leq 
\prod_{j=1}^m |\bs|_1^{s_j/2} = |\bs|_1^{|\bs|_1/2}. 
\end{equation}
By using this estimate and  Lemma \ref{lemma: |s|_1, |s|_0}, we can continue the estimation of 
	$\norm{\Phi_{\Lambda(\xi)} v}{{L_2(\RRm \setminus B^m_\omega,X,\gamma)}}$	as
\begin{align*}
\norm{\Phi_{\Lambda(\xi)} v}{{L_2(\RRm \setminus B^m_\omega,X,\gamma)}}
&
\leq 
Cm  	\sum_{\bs\in \Lambda(\xi)}
\|	v_{\bs}\|_X
(4\omega)^{\frac{m_1}{2}}\exp{(-K\omega)} m_1^{\frac{m_1}{2}}
\\
&  \leq 
Cm  |\Lambda(\xi)|^{1/2}	\bigg(\sum_{\bs\in \Lambda(\xi)}
\|	v_{\bs}\|_X^2\bigg)^{1/2}
(4\omega)^{\frac{m_1}{2}} \exp{(-K\omega)} m_1^{\frac{m_1}{2}}
\\
&
\leq Cm \xi^{1/2} (4\omega)^{\frac{m_1}{2}} \exp{(-K\omega)} m_1^{\frac{m_1}{2}},
\end{align*}
where $m_1=m_1(\xi)$. We have from  the inequality $\frac{1}{\theta q} \le \frac{1}{4}$ and Lemma \ref{lemma: |s|_1, |s|_0} that 
$m_1 \le K_{q,\theta} \xi^{1/4}$, and from Lemma \ref{lemma: m(xi)}  that  $m\leq K_q\xi$. Taking 
account of the choice of $\omega$, we derive the estimate
	\begin{align*}
\| \Phi_{\Lambda(\xi)} v  \|_{{L_2(\RRm \setminus B^m_\omega,X,\gamma)}}
	&
	\leq 
	C \xi^{3/2} (4{K_{q,\theta}}\xi)^{\frac{{K_{q,\theta}}\xi^{1/4}}{2}}  ({K_{q,\theta}}\xi^{1/4})^{\frac{{K_{q,\theta}}\xi^{1/4}}{2}} \exp({-K {K_{q,\theta}}\xi}),
	\end{align*}	
which implies \eqref{norm-g_Lambda}. This completes the proof of the claim (vi).
\\
\\
{\bf Step 3.} 
In this step, we  prove the claims (iv) and (v) in Theorem \ref{thm1-dnn}. 
Namely, we prove that for the size of 
$\bphi_{\Lambda(\xi)}$, 
\begin{align}\label{W} 
W\big(\bphi_{\Lambda(\xi)} \big)
\le C\xi\log \xi,
\end{align}
and  for the depth of $\bphi_{\Lambda(\xi)}$,
\begin{align}	\label{L}
	L\big(\bphi_{\Lambda(\xi)} \big)
	\le 
	C\xi^{1/\theta q} \log\xi,
\end{align}
where the positive constants $C$ are independent of $v$ and $\xi$.

By  Lemma \ref{lem:parallel} and \eqref{W(phi_s,ell)} 
the size of  $\bphi_{\Lambda(\xi)} $ is estimated as
\begin{align} \label{wflambda1}
W\big(\bphi_{\Lambda(\xi)} \big)
= 
\sum_{(\bs,\bell)\in \Lambda^*(\xi)} L(\phi_{\bs,\bell})
\leq 
C
\sum_{(\bs,\bell)\in \Lambda^*(\xi)}
\brac{ 1+ \abs{\bs}_1\log \delta_\bs^{-1}}
\end{align}	
where
\begin{equation}\label{lambda star}
	\Lambda^*(\xi) 
	:=
	\left\{
	(\bs,\bell)\in \FF \times \FF: \bs\in \Lambda(\xi)
	\
	\text{and }
	\boldsymbol{0}\leq \bell \leq \bs
	\right\}.
\end{equation}

From \eqref{eq-Absbell} and the inequality $\log \xi \ge 1$ we derive that
\begin{align} \label{log delta}
\log(\delta_\bs^{-1})
&
\le 
  C \bigg(
\log\xi 
+\log p_\bs(1)+ \abs{\bs}_1 \log \omega
+ \log\brac{\max_{\boldsymbol{0}\le \bell\le \bs} |a_{\bell}|
}\bigg).
\end{align}
Noting that $\log \omega \ge \log \xi$ and $|\bs|_1^2 \ge |\bs|_1$ for all $\bs \in \FF$, we obtain
\begin{equation}\label{eq-log-Absbell0}
\begin{aligned}
\sum_{(\bs,\bell)\in \Lambda^*(\xi)}\brac{ 1 + \abs{\bs}_1\log \delta_\bs^{-1}}
&
\le 
C\Bigg(   \sum_{(\bs,\bell)\in \Lambda^*(\xi)}
\abs{\bs}_1 \log p_\bs(1)
\\
&
+
 \log \omega
\sum_{(\bs,\bell)\in \Lambda^*(\xi)}
\abs{\bs}_1^2
+
\sum_{(\bs,\bell)\in \Lambda^*(\xi)}
\abs{\bs}_1 \log\brac{\max_{\boldsymbol{0}\le \bell\le \bs} |a_{\bell}|
}\Bigg).
\end{aligned}
\end{equation}
To estimate the terms in the right-hand side we need the following auxiliary assertion. 
 Let $\tau\geq 0$, $0<q<\infty$, and $\Lambda^*(\xi)$ be defined in \eqref{lambda star}.
	Assume
	$\brac{p_{\bs}\brac{\frac{\tau+1}{q}}\sigma_{\bs}^{-1}}_{\bs\in \FF}\in \ell_q(\FF)$. 
	Then there holds 
	\begin{align} \label{sum_Lambda^*(xi)}
		\sum_{(\bs,\bell)\in \Lambda^*(\xi)}
		p_{\bs}(\tau) 
		\leq
		C\xi.
	\end{align}
Indeed, from the definition of the set $\Lambda(\xi)$ and the functions $p_{\bs}(\tau)$ we derive 
\begin{align}
	\sum_{(\bs,\bell)\in \Lambda^*(\xi)}
	p_{\bs}(\tau) 
	&
	=
	\sum_{\bs\in \Lambda(\xi)}\sum_{\bell = \boldsymbol{0}}^\bs 
	p_{\bs}(\tau) 
	\leq 
	\xi
	\sum_{\sigma_{\bs}^{-q}\xi\ge 1}\sum_{\bell = \boldsymbol{0}}^\bs 
	p_{\bs}(\tau) 
	\sigma_{\bs}^{-q}
	\notag
	\\
	&=
	\xi
	\sum_{\sigma_{\bs}^{-q}\xi\ge 1}
	\brac{\prod_{j=1}^m(1+s_j)}
	p_{\bs}(\tau) 
	\sigma_{\bs}^{-q}
	\leq 
	\xi
	\sum_{\bs\in \FF}
	p_{\bs}(\tau+1)
	\sigma_{\bs}^{-q}
	\leq 
	C\xi.
	\notag
\end{align} 

Observe that by the definitions \eqref{[p_s]}  we have $\log p_\bs(1) \le |\bs|_1$ and $|\bs|_1^k \le p_\bs(k)$ for $k \in \NN$. Hence, for the first and second terms on the right-hand side of \eqref{eq-log-Absbell0}, since 
$\brac{p_{\bs}\big(\frac{4}{q}\big)
	\sigma_{\bs}^{-1}}_{\bs\in \FF}\in \ell_q(\FF)$ from \eqref{sum_Lambda^*(xi)} we derive that
\begin{equation}\label{eq-first-term}
\sum_{(\bs,\bell)\in \Lambda^*(\xi)}
 \abs{\bs}_1\log p_\bs(1) 
\leq 
\sum_{(\bs,\bell)\in \Lambda^*(\xi)}
\abs{\bs}_1^2
\leq 
\sum_{(\bs,\bell)\in \Lambda^*(\xi)} p_{\bs}(2) \leq C\xi 
\end{equation}
and
\begin{equation}\label{eq-second-term}
 \log \omega\sum_{(\bs,\bell)\in \Lambda^*(\xi)} \abs{\bs}_1^2 
 \leq 
  \log \omega \sum_{(\bs,\bell)\in \Lambda^*(\xi)} p_{\bs}(2) 
  \leq 
  C\xi  \log \omega
\le 
C\xi \log\xi, 
\end{equation}
where in the last inequality we note that $\omega = 
\lfloor K_{q,\theta}\xi \rfloor$, see~\eqref{m,omega}.
Now we turn to the third term in \eqref{eq-log-Absbell0}. 
The inequalities \eqref{eq-sum-abell}  imply
\begin{align*}
	\log\brac{\max_{\boldsymbol{0}\le \bell\le \bs} |a_{\bell}|
	}
	&
	\le 
	\log\Bigg(\prod_{j=1}^m  s_j!
	\Bigg)
	\leq 
	\sum_{j=1}^m
	\log (s_j!)
	\leq 
	\sum_{j=1}^m
	s_j^2
	\leq 
	p_{\bs}(2).
\end{align*}
Using \eqref{sum_Lambda^*(xi)} again we  obtain
\[
\sum_{(\bs,\bell)\in \Lambda^*(\xi)}
\abs{\bs}_1 
p_{\bs}(2)
\le 
\sum_{(\bs,\bell)\in \Lambda^*(\xi)}
p_{\bs}(3)
\le 
C\xi, 
\]
since 
$\brac{p_{\bs}\big(\frac{4}{q}\big)
	\sigma_{\bs}^{-1}}_{\bs\in \FF}\in \ell_q(\FF)$.
This together with \eqref{eq-first-term} and \eqref{eq-second-term}
yields
\begin{align} \label{sum_{(bs,bell) in Lambda^*(xi)}}
\sum_{(\bs,\bell)\in \Lambda^*(\xi)}\brac{ 
	1 
	+ 
	\abs{\bell}_1
	\log 
	\delta_\bs^{-1}
} 
\le 
C\xi\log\xi,
\end{align}
which combined with \eqref{wflambda1} gives \eqref{W}.

We now prove \eqref{L}.
By  Lemma \ref{lem:parallel} and \eqref{L(phi_s,ell)} the depth of  $\bphi_{\Lambda(\xi)} $ is bounded as
\begin{align*}
L\big(\bphi_{\Lambda(\xi)} \big)
= 
\max_{(\bs,\bell)\in \Lambda^*(\xi)} L(\phi_{\bs,\bell})
&
\leq 
C
\max_{(\bs,\bell)\in \Lambda^*(\xi)}
\brac{ 1	+ \log \abs{\bs}_1	\log \delta_\bs^{-1}}.
\end{align*}	
Due to \eqref{log delta}, this inequality can be modified  as
\begin{align} \label{L-flambda2}
L\big(\bphi_{\Lambda(\xi)} \big)
&
\leq 
C
\max_{\bs\in \Lambda(\xi)}
\brac{ 
	\log \abs{\bs}_1
}\max_{(\bs,\bell)\in \Lambda^*(\xi)}
\brac{ 
\log  \delta_\bs^{-1}
}.
\end{align}	
From Lemma \ref{lemma: |s|_1, |s|_0}  we obtain
\begin{align*}
\max_{\bs\in \Lambda(\xi)}\brac{ \log \abs{\bs}_1}
\le 
C \log \xi.
\end{align*}	
We have by \eqref{log delta} that
\begin{equation}\label{L-eq-log-Absbell}
\begin{aligned}
\max_{(\bs,\bell)\in \Lambda^*(\xi)}\brac{ \delta_\bs^{-1}}
&
\le 
C \brac{
	\log \xi
+ \max_{\bs\in \Lambda(\xi)}
 \log p_\bs(1)
+
\log (2\omega)
\max_{\bs\in \Lambda(\xi)}
\abs{\bs}_1 
+
\max_{\bs\in \Lambda(\xi)}
 \log\brac{\max_{\boldsymbol{0}\le \bell\le \bs} |a_{\bell}|}
}.
\end{aligned}
\end{equation}
For the second  and third terms on the right-hand side, we have by the well-known inequality $\log p_\bs(1) \le |\bs|_1$ and Lemma \ref{lemma: |s|_1, |s|_0},
\begin{equation*}\label{L-eq-first-term}
 \max_{\bs\in \Lambda(\xi)} \log p_\bs(1) 
\leq 
\max_{\bs\in \Lambda(\xi)} \abs{\bs}_1  \leq C\xi^{1/\theta q} 
\end{equation*}
and
\begin{equation*}\label{L-eq-second-term}
\log (2\omega)\max_{\bs\in \Lambda(\xi)}
\abs{\bs}_1
 \leq 
 C \xi^{1/\theta q}  \log \xi.
\end{equation*}
Now we turn to the fourth term in \eqref{L-eq-log-Absbell}. 
From \eqref{sum-a_bell} it follows that
\begin{align*}
\log\brac{\max_{\boldsymbol{0}\le \bell\le \bs} |a_{\bell}|
}
&
\le 
\log \brac{|\bs|_1^{|\bs|_1}}
= 
|\bs|_1 \log |\bs|_1.
\end{align*}
Hence,
\[
\max_{(\bs,\bell)\in \Lambda^*(\xi)}
\log\brac{\max_{\boldsymbol{0}\le \bell\le \bs} |a_{\bell}|}
\le 
\max_{\bs\in \Lambda(\xi)}
\brac{\abs{\bs}_1 \log \abs{\bs}_1}
\le 
C\xi^{1/\theta q} \log \xi. 
\]
This together with \eqref{L-flambda2}--\eqref{L-eq-second-term}
yields \eqref{L}.

The proof of Theorem \ref{thm1-dnn} is complete. 	
\hfill
\end{proof}

\begin{theorem}\label{thm1-dnn-2}
	Let  $v \in \Ll_2(X)$ satisfy Assumption A$_{q,\theta}$ with
	$\theta \ge 4/q$.  Then
	for every integer $n > 2$, we can  construct  a deep ReLU neural network $\bphi_{\Lambda(\xi_n)}:= (\phi_\bs)_{\bs \in \Lambda(\xi_n)}$ on 
	$\RR^m$ with $m:=\left\lfloor K _q\frac{n}{\log n} \right\rfloor$ for some positive constant $K_q$, having the following properties. 
	\begin{itemize}
		\item [{\rm (i)}]
		The deep ReLU neural network $\bphi_{\Lambda(\xi_n)}$ is independent of $v$;
		\item [{\rm (ii)}] 
		The input and output dimensions of $\bphi_{\Lambda(\xi_n)}$ are at most $m$;	
			\item [{\rm (iii)}] 
			The  components $\phi_{\bs}$, $\bs\in \Lambda(\xi_n)$, of $\bphi_{\Lambda(\xi_n)}$ are deep ReLU neural networks on 
			$\RR^{|\nu_\bs|}$ with ${|\nu_\bs|}  \le  C_\delta n^\delta$, having support contained in the supercube $[-T,T]^{|\nu_\bs|}$, 
			where $T:= C_\delta \sqrt{\frac{n}{\log n}}$ and $\delta:= \frac{1}{q \theta}$ ;
		\item [{\rm (iv)}] 
		$W\big(\bphi_{\Lambda(\xi_n)}\big) \le n $;
		\item [{\rm (v)}] 
		$L\big(\bphi_{\Lambda(\xi_n)}\big)\le C_\delta n^\delta$;
		\item [{\rm (vi)}]
		The approximation of $v$ by  $\Phi_{\Lambda(\xi_n)}v$ defined as in Theorem \ref{thm1-dnn}(iv),  gives the error estimates 
		$$
		\|v - \Phi_{\Lambda(\xi_n)} v \|_{\Ll_2(X)} 
		\leq C m^{-1/q}
		\leq C \left(\frac{n}{\log n}\right)^{-1/q}.
		$$
	\end{itemize}
Here the positive constants $C = C_{M,q,\bsigma,\theta}$ are independent of $v$ and $n$.	
\end{theorem}
\begin{proof} For a given integer $n > 1$, we choose $\xi_n \ge 2$ as the maximal number satisfying the inequality $C \xi_n \log \xi_n \le n$, where $C$ is  the positive constant in the claim (iii) of Theorem~\ref{thm1-dnn}.  It is easy to verify that there exists a positive constant $C_1$ independent of $n$ such that 
	\begin{equation} \label{<n<}
			C \xi_n \log \xi_n \le n \le C_1\xi_n \log \xi_n.
	\end{equation}
Hence, $C_2 \log \xi_n \le \log n \le C_3 \log \xi_n$ for some positive constants $C_2, C_3$. From the last inequalities and \eqref{<n<} we derive that
there exist positive constants $C_4$ and $C_5$ independent of $n$ such that 
	$$
	C_4 \frac{n}{\log n} \le  \xi_n \le C_5\frac{n}{\log n}. 
	$$
	From Theorem \ref{thm1-dnn} with $\xi=\xi_n$ we deduce the desired results.
	\hfill
\end{proof}

\subsection{Application to parameterized elliptic PDEs with lognormal inputs}
\label {lognormal inputs}

In this section, we apply the results in the previous section to deep ReLU neural network approximation of  the solution $u(\by)$ to the parametric elliptic PDEs \eqref{ellip} with lognormal inputs \eqref{lognormal}.
This is  based on a weighted $\ell_2$-summability of  the series $(\|u_\bs\|_V)_{\bs \in \FF}$ in  following lemma which  conbines \cite[Theorems 3.3 and 4.2]{BCDM17} and \cite[Lemma 5.3]{Dung21}.
\begin{lemma}\label{lemma[bcdm]} Let $\theta$ be an 
	arbitrary nonnegative number and  $(p_\bs(\theta))_{\bs \in 
		\FF}$ the sequence given as in \eqref{[p_s]}.
	Let $0 < q <\infty$ and	
$(\rho_j) _{j \in \NN}$ be an increasing sequence of  positive 
numbers such that  $(\rho_j^{-1}) _{j \in 
	\NN}$ belongs to   $\ell_q(\NN)$.  	Assume further that 
	\begin{equation} \nonumber
	\left\| \sum _{j \in \NN} \rho_j |\psi_j| \right\|_{L_\infty(D)} 
	<\infty\;.
	\end{equation} 
	Then we have that for any $\eta \in \NN$,
		\begin{equation*} 
		\sum_{\bs\in\FF} (\sigma_{\bs} \|u_\bs\|_V)^2 <\infty\ \ \text{with} \ \  \big(p_\bs(\theta)\sigma_\bs^{-1}\big)_{\bs\in \FF} \in {\ell_{q}}(\FF),
	\end{equation*}
where
	\begin{equation} \label{sigma_s}
	\sigma_{\bs}^2:=\sum_{\|\bs'\|_{\ell_\infty(\NN)}\leq \eta}{\bs\choose \bs'} \prod_{j \in \NN}\rho_j^{2s_j'}.
	\end{equation}
\end{lemma}

This weighted $\ell_2$-summability result leads to significant improvements of the convergence rate in the case when the component functions $\psi_j$ have limited overlaps such as splines, finite elements or wavelet bases (for details, see \cite{BCDM17}).
Our result for the solution $u$ to the parametric elliptic PDEs \eqref{ellip} with lognormal inputs \eqref{lognormal}  is read as follows.

\begin{theorem}\label{thm-PDE-lognormal-dnn}
	Under the assumptions of Lemma 
\ref{lemma[bcdm]},  let  $0 < q < \infty$ and 
$\delta$ be arbitrary positive number. Then
	for every integer $n > 2$, we can  construct  a deep ReLU neural network $\bphi_{\Lambda(\xi_n)}:= (\phi_\bs)_{\bs \in \Lambda(\xi_n)}$ on 
	$\RR^m$ with $m:=\left\lfloor K \frac{n}{\log n} \right\rfloor$ for some positive constant $K$, having the following properties. 
	\begin{itemize}
			\item [{\rm (i)}]
		The deep ReLU neural network $\bphi_{\Lambda(\xi_n)}$ is independent of $u$;
		\item [{\rm (ii)}] 
		The input and output dimensions of $\bphi_{\Lambda(\xi_n)}$ are at most $m$;	
			\item [{\rm (iii)}] 
			The  components $\phi_{\bs}$, $\bs\in \Lambda(\xi_n)$, of $\bphi_{\Lambda(\xi_n)}$ are deep ReLU neural networks on 
			$\RR^{|\nu_\bs|}$ with ${|\nu_\bs|}  \le  C_\delta n^\delta$, having support contained in the super-cube $[-T,T]^{|\nu_\bs|}$, 
			where $T:= C_\delta \sqrt{\frac{n}{\log n}}$;
		\item [{\rm (iv)}] 
		$W\big(\bphi_{\Lambda(\xi_n)}\big) \le n $;
		\item [{\rm (v)}] 
		$L\big(\bphi_{\Lambda(\xi_n)}\big)\le C_\delta n^\delta$
		\item [{\rm (vi)}]
		The approximation of $u$ by  $\Phi_{\Lambda(\xi_n)}u$ defined as in Theorem \ref{thm1-dnn}(iv),  gives the error estimates 
		$$
		\| u- \Phi_{\Lambda(\xi_n)} u \|_{\Ll_2(V)} 
			\leq C m^{-1/q}
		\leq C \left(\frac{n}{\log n}\right)^{-1/q}.
		$$
	\end{itemize}
	Here the positive constants $C$ and  $C_\delta$ are independent of $u$ and $n$.	
\end{theorem}

\begin{proof} Since $(\rho_j) _{j \in \NN}$ is an increasing sequence of  positive 
		numbers, it is easily seen that $\bsigma =(\sigma_\bs)_{\bs \in \FF}$ is an increasing sequence and that   $\sigma_{\be^{i'}} \le \sigma_{\be^i}$ if $i' 
		< i$. Therefore Assumption A$_{q,\theta}$ is satisfied. To prove the theorem we apply Theorem \ref{thm1-dnn-2} to the solution $u$.  Without loss of generality we can assume that $\delta \le 1/4$. We take first the number 
$\theta := 1/\delta q$ satisfying the inequality $\theta \ge 4/q$, and then choose a number $\eta \in \NN$ satisfying the inequality $\eta > \frac{2(\theta + 1)}{q}$. By using Lemma~\ref{lemma[bcdm]} one can check that  for $X= V$ and the sequence $(\sigma_\bs)_{\bs \in \FF}$ defined as in \eqref{sigma_s}, $u \in \Ll_2(V)$ satisfies the assumptions of Theorem \ref{thm1-dnn-2} which proves the theorem. 
	\hfill
\end{proof}


\subsection{Application to approximation of holomorphic functions}
\label {holomorphic maps}
In this section we show that some holomorphic functions satisfy the assumption and therefore we can explicitly construct deep ReLU neutral network to approximate them. Let us first introduce the concept of  ``$(\bb,\xi,\varepsilon,X)$-holomorphic
	functions''  which has been introduced in \cite{DNSZ22}.

We recall  the concept of ``$(\bb,\xi,\varepsilon,X)$-holomorphic
functions''  which has been introduced in \cite{DNSZ22}. 
For $N\in\N$ and a positive sequence $\bvarrho=(\varrho_j)_{j=1}^N$,  we put
\begin{equation}
	\label{eq:Sjrho}
	\Ss(\bvarrho) := \set{\bz\in \CC^N}{|\mathfrak{Im}z_j| < \varrho_j~\forall j}\qquad\text{and}\qquad
	\Bb(\bvarrho) := \set{\bz\in\CC^N}{|z_j|<\varrho_j~\forall j}.
\end{equation}

	Let $X$ be a complex separable Hilbert space,
	$\bb=(b_j)_{j\in\N}$ a positive sequence, and $\xi>0$, $\varepsilon>0$.
	For $N\in\N$ we say that a positive sequence $\brho=(\rho_j)_{j=1}^N$ is
	\emph{$(\bb,\xi)$-admissible} if
	\begin{equation}\label{eq:adm}
		\sum_{j=1}^N b_j\varrho_j\leq \xi\,.
	\end{equation}
	A function $v\in \Ll_2(X)$ is called
	$(\bb,\xi,\varepsilon,X)$-holomorphic if
	\begin{enumerate}
		\item[{\rm (i)}]\label{item:hol} for every $N\in\N$ there exists
		$v_N:\R^N\to X$, which, for every $(\bb,\xi)$-admissible
		$\bvarrho$, admits a holomorphic extension
		(denoted again by $v_N$) from $\Ss(\bvarrho)\to X$; furthermore,
		for all $N<M$
		\begin{equation}\label{eq:un=um}
			v_N(y_1,\dots,y_N)=v_M(y_1,\dots,y_N,0,\dots,0)\qquad\forall (y_j)_{j=1}^N\in\R^N,
		\end{equation}
		
		\item[{\rm (ii)}]\label{item:varphi} for every $N\in\N$ there exists
		$\varphi_N:\R^N\to\R_+$ such that
		$\norm{\varphi_N}{L^2(\R^N;\gamma)}\le\varepsilon$ and
		\begin{equation*} \label{ineq[phi]}
			\sup_{\text{$\brho$ is $(\bb,\xi)$-adm.}}~\sup_{\bz\in
				\Bb(\bvarrho)}\norm{v_N(\by+\bz)}{X}\le
			\varphi_N(\by)\qquad\forall\by\in\R^N,
	\end{equation*}
	\item[{\rm (iii)}]\label{item:vN} with $\tilde v_N:\RRi\to X$ defined by
	$\tilde v_N(\by) :=v_N(y_1,\dots,y_N)$ for $\by\in \RRi$ it holds
	\begin{equation*}
		\lim_{N\to\infty}\norm{v-\tilde v_N}{\Ll_2(X)}=0.
	\end{equation*}
\end{enumerate}

The following key result on  weighted $\ell_2$-summability of $(\bb,\xi,\varepsilon,X)$-holomorphic functions has been proven in \cite[Corollary 4.9,]{ DNSZ22}.
\begin{theorem} \label{thm:Holom-AssumpA} Let $v$ be
	$(\bb,\xi,\varepsilon,X)$-holomorphic for some $\bb\in \ell_p(\N)$ with $0< p <1$.  Let $\eta\in\N$ and let the sequence $\brho=(\rho_j)_{j \in \NN}$ be defined by
	$$
	\rho_j:=b_j^{p-1}\frac{\xi}{4\sqrt{\eta!}} \norm{\bb}{\ell_p(\NN)}.
	$$
Assume that $\bb$ is a decreasing sequence and that $b_j^{p-1}\frac{\xi}{4\sqrt{\eta!}} \norm{\bb}{\ell_p(\NN)}>1$ for all $j\in \NN$.
	Then $v$ satisfies Assumption A$_{q,\theta}$ for $q := \frac{p}{1-p}$,
	$\bsigma =\bsigma(\brho,\eta):=(\sigma_\bs)_{\bs \in \FF}$ given by \eqref{sigma_s}, and 
	$M:= \varepsilon C_{\bb,\xi}$ with some positive  constant $ C_{\bb,\xi}$.
		\end{theorem}
The condition $\bb$ is a decreasing sequence implies $\brho$ is an increasing sequence and therefore  $\sigma_{\be^{i'}} \le \sigma_{\be^i}$ if $i' 
	< i$. Moreover   $\rho_j=b_j^{p-1}\frac{\xi}{4\sqrt{\eta!}} \norm{\bb}{\ell_p(\NN)}>1$ for all $j\in \NN$ and 	$\bsigma =\bsigma(\brho,\eta):=(\sigma_\bs)_{\bs \in \FF}$ given by \eqref{sigma_s} imply that $\big(p_\bs(\theta)\sigma_\bs^{-1}\big)_{\bs\in \FF} \in {\ell_{q}}(\FF)$ for any $\theta>0$. For a proof we refer the reader again to \cite[Lemma 5.3]{DDK21}.
	
	In the same way as the proof of Theorem \ref{thm-PDE-lognormal-dnn}, from Theorem \ref{thm:Holom-AssumpA} we derive 
	
	\begin{theorem}\label{thm-HolomorphicMaps}
		Let $v$ be
		$(\bb,\xi,\varepsilon,X)$-holomorphic for some $\bb\in \ell_p(\N)$ with $0< p <1$, and let  
		$\delta$ be an arbitrary positive number. Then, with the notations of Theorem \ref{thm:Holom-AssumpA}, 
		for every integer $n > 2$, we can  construct  a deep ReLU neural network $\bphi_{\Lambda(\xi_n)}:= (\phi_\bs)_{\bs \in \Lambda(\xi_n)}$ on 
		$\RR^m$ with $m:=\left\lfloor K \frac{n}{\log n} \right\rfloor$ for some positive constant $K$, having the following properties. 
		\begin{itemize}
			\item [{\rm (i)}]
			The deep ReLU neural network $\bphi_{\Lambda(\xi_n)}$ is independent of $v$;
			\item [{\rm (ii)}] 
			The input and output dimensions of $\bphi_{\Lambda(\xi_n)}$ are at most $m$;	
			\item [{\rm (iii)}] 
			The  components $\phi_{\bs}$, $\bs\in \Lambda(\xi_n)$, of $\bphi_{\Lambda(\xi_n)}$ are deep ReLU neural networks on 
				$\RR^{|\nu_\bs|}$ with ${|\nu_\bs|}  \le  C_\delta n^\delta$, having support contained in the super-cube $[-T,T]^{|\nu_\bs|}$, 
				where $T:= C_\delta \sqrt{\frac{n}{\log n}}$;
			\item [{\rm (iv)}] 
			$W\big(\bphi_{\Lambda(\xi_n)}\big) \le n $;
			\item [{\rm (v)}] 
			$L\big(\bphi_{\Lambda(\xi_n)}\big)\le C_\delta n^\delta$;
			\item [{\rm (vi)}]
			The approximation of $v$ by  $\Phi_{\Lambda(\xi_n)}v$ defined as in Theorem \ref{thm1-dnn}(iv),  gives the error estimates 
			$$
			\|v - \Phi_{\Lambda(\xi_n)} v\|_{\Ll_2(X)} 
			\leq C m^{- (1/p - 1)}
			\leq C \left(\frac{n}{\log n}\right)^{- (1/p - 1)}.
			$$
		\end{itemize}
		Here the positive constants $C$, $C_\delta$ and  $C'_\delta$ are independent of $v$ and $n$.	
	\end{theorem}

We notice some important examples of $(\bb,\xi,\varepsilon,X)$-holomorphic functions which are solutions to parametric PDEs equations and which were studied in \cite{DNSZ22}. 

Formally, 
replacing $\by=(y_j)_{j \in \NN}\in \RRi$ in the coefficient $a(\by)$ in \eqref{lognormal} by
$\bz=(z_j)_{j \in \NN}=(y_j+i \xi_j)_{j \in \NN}\in \CC^\infty$, 
the real part of $a(\bz)$ is
\begin{equation} \label{Re(a)} 
	\mathfrak{R}[a(\bz)] =
\exp\Bigg({\sum_{j \in \NN} y_j\psi_j(\bx)}\Bigg) \cos\Bigg(\sum_{j
	\in \NN} \xi_j\psi_j(\bx)\Bigg)\,.
\end{equation}
We find that $\mathfrak{R}[a(\bz)]>0$ if
$$
\bigg\|\sum_{j \in \NN} \xi_j\psi_j \bigg\|_{L_\infty(D)} <
\frac{\pi}{2}.
$$
This observation motivate the study
of the analytic continuation of the solution map $\by \mapsto u(\by)$
to $\bz \mapsto u(\bz)$ for complex parameters
$\bz = (z_j)_{j \in \NN}$ where each
$z_j$ lies in the strip
\begin{equation} \label{eq:DefSjrho} \mathcal{S}_j (\brho):= \{ z_j\in
\CC\,: |\mathfrak{Im}z_j| < \rho_j\}
\end{equation}
and where $\rho_j>0$ and
$\brho=(\rho_j)_{j \in \NN} \in (0,\infty)^\infty$ is any sequence of
positive numbers such that
\begin{equation*} \label{k-011} \Bigg\|\sum_{j\in \NN} \rho_j
|\psi_j|\Bigg\|_{L_\infty(D)} < \frac{\pi}{2}\,.
\end{equation*} 
For further detail of this continuation we refer to \cite[Proposition 3.8]{DNSZ22}.

In general, let $b(\by)$ be defined as in \eqref{lognormal} and $\Vv$  a holomorphic map from an open set in  $L_\infty(D)$ to $X$. Then function compositions of the type 
$$
v(\by)= \Vv(\exp(b(\by)))
$$
 are $(\bb,\xi,\varepsilon,X)$-holomorphic under certain conditions \cite[Proposition 4.11]{DNSZ22}. This allows us to apply Theorem \ref{thm-HolomorphicMaps} for deep ReLU neural network approximation of solutions  $v(\by)= \Vv(\exp(b(\by)))$  as $(\bb,\xi,\varepsilon,X)$-holomorphic functions on various function spaces $X$, to a wide range of parametric and stochastic PDEs with lognormal inputs. Such function spaces $X$ are high-order regularity spaces $H^s(D)$ and corner-weighted Sobolev (Kondrat'ev) spaces $K^s_\varkappa(D)$ ($s \ge 1$) for  the parametric elliptic PDEs \eqref{ellip} with lognormal inputs \eqref{lognormal}; spaces of solutions to linear parabolic PDEs with lognormal inputs \eqref{lognormal}; spaces of solutions to linear elastostics equations with lognormal modulus of elasticity; spaces of solutions to Maxwell equations with lognormal permittivity.
 
 After the first version of the present paper appeared in ArXiv website, we have been informed about the paper \cite{SZ2021} on deep ReLU neural network approximation of holomorphic functions,  in a private communication with its authors. Subsection \ref{holomorphic maps} has been added in the second ArXiv version after  the paper \cite{SZ2021}  appeared.  The results on convegence rate of this subsection improved the results  of  \cite{SZ2021} by the help of using deep ReLU neural networks connecting neurons in a layer with neurons in preceding layers (see Section \ref{Deep ReLU neural networks}).

\section{Approximation in Bochner spaces with Jacobi measure}
\label{Approximation in Bochner space with Jacobi measure}

The theory of non-adaptive  deep ReLU neural network approximation   in Bochner spaces with  Gaussian measure, which has been discussed  in Section~\ref{Approximation in Bochner space with Gaussian measure} can be generalized and extended to other situations. In this section, we investigate  non-adaptive methods of deep ReLU neural network approximation  in Bochner spaces with Jacobi probability measure.  Functions to be approximated  satisfy a weighted $\ell_2$-summability  of their Jacobi gpc expansion coefficients. We
	construct such methods and prove the convergence rate of the approximation by them. 
	These menthods are constructed via the  truncated Jacobi gpc expansion of functions.
	The results  are then applied to the
	approximation of the solutions  to  parametric elliptic PDEs \eqref{p-ellip} with affine inputs \eqref{affine}.

\subsection{Approximation by truncations of the Jacobi gpc expansion}
\label{Approximation by truncations of the Jacobi gpc expansion}

For given $a,b > -1$, we consider  the infinite tensor product of the Jacobi probability measure on $\IIi$
\begin{equation} \nonumber
\rd \nu_{a,b}(\by):=\bigotimes_{j \in \NN} \delta_{a,b}(y_j)\,\rd y_j,
\end{equation}
where
\[
\delta_{a,b}(y):=c_{a,b}(1-y)^a(1+y)^b, \quad
c_{a,b}:=\frac{\Gamma(a+b+2)}{2^{a+b+1}\Gamma(a+1)\Gamma(b+1)}.
\] 
If $v \in \Ll_2(X):= L_2(\IIi, X, \nu_{a,b})$ for a separable Hilbert space $X$,
we consider  the orthonormal Jacobi gpc expansion of  $v $  of the form
\begin{equation} \label{J-series}
v = \sum_{\bs\in\FF} v_\bs J_\bs(\by), 
\end{equation}
where
\begin{equation} \nonumber
J_\bs(\by)=\bigotimes_{j \in \NN}J_{s_j}(y_j),\quad v_\bs:=\int_{\IIi} v(\by)J_\bs(\by) \rd\nu_{a,b} (\by),
\end{equation}
and  $(J_k)_{k\geq 0}$ is the sequence of Jacobi polynomials on $\II := [-1,1]$ 
normalized with respect to the Jacobi probability measure, i.e., 
$
\int_{\II} |J_k(y)|^2 \delta_{a,b}(y) \rd y =1.
$
One has the Rodrigues' formula
\[
J_k(y ) 
\ = \
\frac{c_k^{a,b}}{k! 2^k}(1-y)^{-a}(1+y)^{-b} \frac{\rd^k}{\rd y^k} 
\left((y^2-1)^k(1-y)^a(1+y)^b\right),
\]
where $c_0^{a,b}:= 1$ and
\begin{equation}\label{eq-cabk}
c_k^{a,b}
:= \
\sqrt{\frac{(2k+a+b+1)k! \Gamma(k+a+b+1) \Gamma(a+1) \Gamma(b+1)}
	{\Gamma(k+a+1)\Gamma(k+b+1)\Gamma(a+b+2)}}, \ k \in \NN.
\end{equation}
Examples corresponding to the values $a=b=0$  are the family of the Legendre polynomials, and to the values $a=b=-1/2$  the family of the Chebyshev polynomials. 

Throughout this section, if $f$ is a function on $\IIm$ taking values in a Hilbert space $X$, then $f$ has an extension to $\II^{m'}$ for $m' > m$ or to $\IIi$ which is denoted again by $f$, by the formula 
	$f(\by)= f \brac{(y_j)_{j=0}^m}$ for $\by = (y_j)_{j =1}^{m'}$ or $\by = (y_j)_{j \in \NN}$, respectively. 

\medskip
\noindent
{\bf Assumption B} \ 	Let $0 < q <\infty$, $c_k^{a,b}$ be defined as in \eqref{eq-cabk} and let $(\delta_j)_{j\in \NN}$ be a sequence of numbers strictly larger than 1 such that $(\delta_j^{-1}) _{j \in \NN} \in \ell_q(\NN)$. For $v \in \Ll_2(X)$  represented by the series \eqref{J-series}, 
there exists a sequence of positive numbers
 	$ (\rho_j)_{j\in \mathbb{N}}$ such that $c_k^{a,b} \rho_j^{-k} \le  \delta_j^{-k}$ for $k,j\in \NN$ and 
\begin{equation*} \label{L-sigma-summability}
\left(\sum_{\bs\in\FF} (\sigma_\bs \|v_\bs\|_{X})^2\right)^{1/2} \ \le M \ <\infty,
\end{equation*}
where
\begin{equation} \label{sigma_bs}
\sigma_\bs  := c_\bs^{-1} \prod_{j \in \NN} \rho_j^{s_j}, \quad c_\bs:= \prod_{j \in \NN} c_{s_j}^{a,b}.
\end{equation}

In this subsection, for the function $v \in \Ll_2(X)$ represented by the series \eqref{J-series} and the sequence 
$(\sigma_\bs)_{\bs \in \FF}$ given as in  \eqref{sigma_bs}, we consider the approximation of $v$ by the truncation
\begin{equation*} \label{L-S_{Lambda(xi)}v}
	S_{\Lambda(\xi)} v 
	:= \ 
	\sum_{\bs\in {\Lambda(\xi)}} v_\bs J_\bs,
\end{equation*}
where $\Lambda(\xi)$ is defined by the formula \eqref{Lambda(xi)} for  the sequence   $(\sigma_\bs)_{\bs \in \FF}$ given as in \eqref{sigma_bs}. 

In the same way as	the proof of Lemma \ref{lemma[L_2-approx]}, we prove 

\begin{lemma}\label{lemma[L_2-approx-J]} 
	For every $v \in \Ll_2(X)$ satisfying Assumption B and
	for every $\xi >1$, there holds
	\begin{equation} \nonumber
		\|v- S_{\Lambda(\xi)}v\|_{\Ll_2(X)} \leq M \xi^{- 
			1/q}.
	\end{equation}
\end{lemma}

We will need some auxiliary results for further use. The following lemma is a direct consequence of \cite[Lemma 6.2]{Dung21}.
\begin{lemma} \label{bcdmJ}
	Let $0 < q <\infty$ and $\theta$ and  $\lambda$ 
	be arbitrary  nonnegative real numbers. Assume that $\brho = (\rho_j)_{j\in \NN}$ be a sequence of numbers strictly larger than 1 such that 
	$(\rho_j^{-1}) _{j \in \NN} \in \ell_{q}(\NN)$.  Then  for  
	the sequences $(\sigma_\bs)_{\bs \in \FF}$ and 
	$(p_\bs(\theta,\lambda))_{\bs \in \FF}$ given as in  
	\eqref{sigma_bs} and   \eqref{[p_s]}, respectively, we have
	\begin{equation} \nonumber
		\sum_{\bs \in \FF} p_\bs(\theta,\lambda) \sigma_\bs^{-q}< \infty.
	\end{equation}
\end{lemma}


\begin{lemma} \label{lemma:m_1 < log xi}
	Let $0 < q <\infty$, $c_k^{a,b}$ be defined as in \eqref{eq-cabk} and let $(\delta_j)_{j\in \NN}$ be a sequence of numbers strictly larger than 1 such that $(\delta_j^{-1}) _{j \in \NN} \in \ell_q(\NN)$. 	
	Assume that there exists a sequence of positive number $(\rho_j) _{j \in \NN}$   such that 
	$c_k^{a,b}\rho_j^{-k}<\delta_j^{-k}$, $k,j\in \NN$.  For  the sequence $(\sigma_\bs)_{\bs \in \FF}$ given as in  
	\eqref{sigma_bs}, let $m_1(\xi)$ be the number 
	defined by \eqref{m_p(xi)}. 
	Then  we have for every $\xi  \ge 2$,
	\begin{equation} \label{m_1<}
		m_1(\xi) \le C \log \xi,
	\end{equation}
	with  the positive constant $C$ independent of $\xi$.
\end{lemma}

\begin{proof}
	The proof relies on Lemmata  \ref{lemma: |s|_1, |s|_0} and \ref{bcdmJ} and a technique from the proof of \cite[Lemma 2.8(ii)]{SZ19}.	
	Fix a number $p$ satisfying $0 < p <q$ and let  the sequence 	
	$(\beta_\bs)_{\bs \in \FF}$ be given by 
	\begin{equation*}
		\beta_\bs^{-1}:= 
		\begin{cases}
			\max(\sigma_\bs^{-1}, j^{-1/p}) \  & {\rm if } \ \bs = \be^j,\\
			\sigma_\bs^{-1} \ &  {\rm otherwise}.
		\end{cases}
	\end{equation*}
	Notice that the sequence 	
	$(\alpha_\bs^{-1})_{\bs \in \FF}$ defined by 
	\begin{equation*}
		\alpha_\bs^{-1}:= 
		\begin{cases}
			j^{-1/p} \  & {\rm if } \ \bs = \be^j,\\
			0 \ &  {\rm otherwise},
		\end{cases}
	\end{equation*}
	belongs to  $\ell_q(\FF)$.
	On the other hand, from  Lemma \ref{bcdmJ} one can see that the sequence   $(\sigma_\bs^{-1})_{\bs \in \FF}$ belongs to  $\ell_q(\FF)$. This implies that the sequence $(\beta_\bs^{-1})_{\bs \in \FF}$ belongs to  $\ell_q(\FF)$. Hence, by Lemma~\ref{lemma: |s|_1, |s|_0} the set
	$\Lambda_\beta(\xi)
	:= \ 
	\big\{\bs \in \FF: \, \beta_{\bs}^{q} \le \xi \big\}$
	is finite. Notice also that  $(\beta_\bs)_{\bs \in \FF}$ is increasing and 
	$\Lambda_\beta(\xi)$ is downward closed. Put $n:= |\Lambda_\beta(\xi)|$. Then the set $\Lambda_\beta(\xi)$ contains $n$ largest elements of $(\beta_\bs)_{\bs \in \FF}$.  Therefore by the construction of $(\beta_\bs)_{\bs \in \FF}$ we have
	\[
	\min_{\bs \in \Lambda_\beta(\xi)} \beta_\bs^{-1}
	= 
	\beta_{\bs_n}^{-1}
	\ge 
	n^{-1/p}.
	\]
	Since  $c_k^{a,b} \rho_j^{-k} \le  \delta_j^{-k}$, $k,j\in \NN$ and $(\delta_j)_{j\in \NN}$ be a sequence of numbers strictly larger than 1 and $(\delta_j^{-1}) _{j \in \NN} \in \ell_q(\NN)$, there exists $\delta<1$ such that $c_k^{a,b} \rho_j^{-k} \le  \delta$ for $ k,j\in \NN$. Therefore have for $r > 1$,
	\[
	\sup_{|\bs|_1 = r} \beta_\bs^{-1}
	= 
	\sup_{|\bs|_1 = r} \sigma_\bs^{-1}
	\le 
	\delta^r.
	\]
	Let $\bar{r}> 1$ be an integer such that $n^{-1/p} > \delta^{\bar{r}}$. Then one can see that
	$$
	\max_{\bs \in \Lambda_\beta(\xi)} |\bs|_1 < \bar{r}.
	$$
	For the function $g(t):= \delta^t$, its inverse is defined as $g^{-1}(x) = \frac{\log x}{\log \delta}$. Hence we get 
	$\bar{r}< g^{-1}(n^{-1/p})$, and consequently,
	$$
	\max_{\bs \in \Lambda_\beta(\xi)} |\bs|_1 < g^{-1}(n^{-1/p}) \le C \log n = C \log |\Lambda_\beta(\xi)|.
	$$
	By Lemma  \ref{lemma: |s|_1, |s|_0}  we obtain the inequality $|\Lambda_\beta(\xi)| \le C \xi$ which together with the inclusion 
	$\Lambda(\xi) \subset \Lambda_\beta(\xi)$ proves \eqref{m_1<}.
	\hfill
\end{proof}	

\begin{lemma} \label{lemma: J_s}
	Let the Jacobi polynomial $J_s$ be written in the form	
	\begin{align}\label{J_s}
		J_s(y)
		=	
		\sum_{\ell=0}^s a_{s,\ell} y^\ell,
	\end{align}
	then
	\begin{equation} \label{J-ineq-coeff}
		\sum_{\ell=0}^{s} |a_{s,\ell}| \leq K_{a+b}6^s.
	\end{equation}
\end{lemma}
\begin{proof}
	It is well-known  that for each $s\in \NN$, the univariate Jacobi polynomial $J_s$
	can be written as 
	\begin{align*} 
		J_s(y)
		&
		=
		\frac{\Gamma(a+s+1)}{s! \Gamma(a+b+s+1)}
		\sum_{m=0}^s \binom{s}{m} \frac{\Gamma(a+b+s+m+1)}{\Gamma(a+m+1)} \left(\frac{y-1}{2}\right)^m,
	\end{align*}
	where $\Gamma$ is the gamma function. From $\frac{y-1}{2}=\frac{y-2+1}{2} $ we see that $|a_{s,\ell}|$ is equal to the  coefficient of $(y-2)^\ell$. Therefore, choosing $y=3$ we get
	\begin{align}\label{sum-coeffJ_s2}
		\sum_{\ell=0}^{s} |a_{s,\ell}| 
		& =
		\frac{\Gamma(a+s+1)}{s! \Gamma(a+b+s+1)}
		\sum_{m=0}^s \binom{s}{m}  \frac{\Gamma(a+b+s+m+1)}{\Gamma(a+m+1)}.
	\end{align}
	Let 
	$$
	B(x,y):= \int_0^1 t^{x-1} (1-t)^{y-1} \rd t = \frac{\Gamma(x) \Gamma(y)}{\Gamma(x+y)}
	$$
	be the beta function. It is decreasing in $x$ and in $y$. Hence for $m \le s$,
	$$
	\frac{\Gamma(a+b+s+m+1)}{\Gamma(a+m+1)} 
	=
	\frac{\Gamma(b+s)}{B(a+m+1,b+s)} 
	\le
	\frac{\Gamma(b+s)}{B(a+s+1,b+s)}
	=
	\frac{\Gamma(a+b+2s+1)}{\Gamma(a+s+1)}. 
	$$
	This   gives
	\begin{align}
		\sum_{\ell=0}^{s} |a_{s,\ell}| 
		& \le
		\frac{\Gamma(a+b+2s+1)}{s!\Gamma(a+s+1)} 
		\sum_{m=0}^s \binom{s}{m}
		=
		2^s\frac{\Gamma(a+b+2s+1)}{s!\Gamma(a+b+s+1)}. 
	\end{align}
	By using Stirling's formula for the gamma  function $\Gamma(x+\alpha)\sim \Gamma (x)x^\alpha$,  we get 
	\begin{align*} 
		\sum_{\ell=0}^{s} |a_{s,\ell}| 
		& \le 
		C \frac{(2e)^s(a+b+s+1)^s}{s^s}
		\le  
		C (2e)^s \left(\frac{a+b+1}{s} + 1\right)^s
		\le K_{a+b}6^s.
		\notag
	\end{align*}	
	\hfill
\end{proof}

\subsection{Approximation by deep ReLU neural networks}
	
	\begin{theorem}\label{L-thm2-dnn}
		Let  $v \in \Ll_2(X)$ satisfy Assumption B.  Then
		for every integer $n \ge 4$, we can  construct  a deep ReLU neural network $\bphi_{\Lambda(\xi_n)}:= (\phi_\bs)_{\bs \in \Lambda(\xi_n)}$ on 
		$\RR^m$ with $m:=\lfloor K \frac{n}{\log n} \rfloor$ for some positive constant $K$, having the following properties. 
		\begin{itemize}
			\item [{\rm (i)}]
			The deep ReLU neural network $\bphi_{\Lambda(\xi_n)}$ is independent of $v$;
			\item [{\rm (ii)}] 
			The input and output dimensions of $\bphi_{\Lambda(\xi_n)}$ are at most $m$;	
			\item [{\rm (iii)}]
			$W\big(\bphi_{\Lambda(\xi_n)}\big) \le n $;
			\item [{\rm (iv)}] 
			$L\big(\bphi_{\Lambda(\xi_n)}\big)\le  C \, \log n \, \log\log n$;
			\item [{\rm (v)}]The  components $\phi_{\bs}$, $\bs\in \Lambda(\xi_n)$, of $\bphi_{\Lambda(\xi_n)}$ are deep ReLU neural networks on 
			$\RR^{|\nu_\bs|}$.
			If $\Phi_{\Lambda(\xi_n)}v$  is	defined as in Theorem \ref{thm1-dnn}(iv) with replacing $\RRi$ by $\IIi$, then the  approximation of $v$ by  $\Phi_{\Lambda(\xi_n)}v$  gives the error estimates 
			$$
			\| v- \Phi_{\Lambda(\xi_n)} v \|_{\Ll_2(X)} 
				\leq C m^{-1/q}
			\leq C \left(\frac{n}{\log n}\right)^{-1/q}.
			$$
		\end{itemize}
		
		Here the positive constants $C$  are independent of $v$ and $n$.
	\end{theorem}

\begin{proof} Similar to the proof of Theorem \ref{thm1-dnn-2}, this theorem is deduced from a counterpart of Theorem \ref{thm1-dnn} for the case $\IIi$. It  states 
that for every $\xi \ge 4$, we can  construct  a deep ReLU neural network $\bphi_{\Lambda(\xi)}:= (\phi_\bs)_{\bs \in \Lambda(\xi)}$ on 
$\IIm$ with $m\leq \lfloor K_q \xi \rfloor$, having the following properties. 
\begin{itemize}
	\item [{\rm (i)}]
	The deep ReLU neural network $\bphi_{\Lambda(\xi)}$ is independent of $v$;
	
	\item [{\rm (ii)}] 
	The input and output dimensions of $\bphi_{\Lambda(\xi)}$ are at most $m$;
	
	\item [{\rm (iii)}] 
	$W\big(\bphi_{\Lambda(\xi)}\big) \le C\xi \log \xi$;
	
	\item [{\rm (iv)}]
	$L\big(\bphi_{\Lambda(\xi)}\big) \le C \log \xi \,\log\log \xi$;

	\item [{\rm (v)}] 
	The  components $\phi_{\bs}$, $\bs\in \Lambda(\xi)$, of $\bphi_{\Lambda(\xi)}$ are deep ReLU neural networks on 
	$\RR^{|\nu_\bs|}$.
	If $\Phi_{\Lambda(\xi)}v$  is	defined as in Theorem \ref{thm1-dnn}(iv) with replacing $\RRi$ by $\IIi$, then the  approximation of $v$ by  $\Phi_{\Lambda(\xi)}v$  gives the error estimate
	$$
	\| v- \Phi_{\Lambda(\xi)}v  \|_{\Ll_2(X)}\leq C\xi^{-1/q}.
	$$
\end{itemize}
	Here the positive constants $C$  are independent of $v$ and $\xi$. 
	
 The proofs of these claims are similar to the proof of Theorem \ref{thm1-dnn}, but  simpler due to the compact property of $\IIi$. To prove the claims (i)--(v) let us follow the steps in that proof.
 
{\bf Step 1.} In this step, we construct a deep ReLU neural network  $\bphi_{\Lambda(\xi)}:= (\phi_\bs)_{\bs \in \Lambda(\xi)}$ on $\RRm$ and prove the claims (i) and (ii). 

 Suppose that $\bphi_{\Lambda(\xi)}$ and therefore, the function $\Phi_{\Lambda(\xi)}$ are already constructed. Due to the inequality
\begin{equation}\label{eq-decompose-J}
		\|v- \Phi_{\Lambda(\xi)} v\|_{\Ll_2(X)}
		\leq \|v-S_{\Lambda(\xi)} v  \|_{\Ll_2(X)}
		+ \| S_{\Lambda(\xi)} v- \Phi_{\Lambda(\xi)} v\|_{\Ll_2(X)},
\end{equation}
the claim (iv) will be proven if we show the bound  $C\xi^{-1/q}$ for the two terms in the right-hand side. This bound has been shown for  first term   as in Lemma \ref{lemma[L_2-approx-J]}.  By Lemma \ref{lemma: m(xi)} $S_{\Lambda(\xi)}v $ can be considered as a function on $\II^m$. Therefore, to prove  the claim (iv) we will construct 
$\bphi_{\Lambda(\xi)}$ on $\IIm$ which is able to approximate $S_{\Lambda(\xi)} v$ with an error bounded by $C\xi^{-1/q}$. 
 
From \eqref{J_s} for each  $\bs\in \FF$  we have 
\begin{align*}
J_{\bs}(\by)
=
\sum_{\bell=\boldsymbol{0}}^{\bs}  a_{\bell}
\by^{\bell},
\end{align*}
where  
$a_{\bell}:= \prod_{j=1}^m a_{s_j,\ell_j}$ and 
$\by^{\bell}:= \prod_{j=1}^m y_j^{\ell_j}$.
Hence, we get for every $\by \in \IIm$,
\begin{align} \label{S_Lmabda(xi)-J}
S_{\Lambda(\xi)}v(\by)
:=
\sum_{\bs\in \Lambda(\xi)} v_{\bs} J_{\bs}(\by)
&	=
\sum_{\bs\in \Lambda(\xi)} 
v_{\bs} 
\sum_{\bell=\boldsymbol{0}}^{\bs} a_{\bell} \prod_{j \in \nu_\bell} y_j^{\ell_j}.
\end{align}
We have
\begin{equation} 
\prod_{j \in \nu_\bell} y_j^{\ell_j}
	= 
	\prod_{j \in \nu_{\bell}} 
	\varphi_1^{\ell_j}\brac{y_j}\prod_{j \in \nu_\bs \setminus \nu_{\bell}} 
	\varphi_0\brac{y_j}=h_{\bs,\bell}(\by), 
	\ \ \by \in \II^{|\nu_\bs|},
\end{equation}
where
\begin{equation*} \label{varphi=J}
	h_{\bs,\bell}(\by) 
	:=
	\prod_{j \in \nu_\bell} \varphi_1^{\ell_j}(y_j)
	\prod_{j \in \nu_\bs\setminus\nu_\bell}
	\varphi_0\brac{y_j}, 
\end{equation*}
$\varphi_0$ and $\varphi_1$ are the piece-wise linear functions defined before Lemma~\ref{lem-product-varphi}.
By applying Lemma \ref{lem-product-varphi} to the product in the right-hand side, for every $\bell$ with  $\boldsymbol{0} \le \bell  \le \bs$,
with 
\begin{align}  \label{delta-J}
\delta_\bs^{-1}
:= 
\xi^{1/q + 1/2} p_\bs(1)
\displaystyle
\sum_{\bell=\boldsymbol{0}}^{\bs}|a_{\bell}|,
\end{align}
there exists a deep  ReLU neural network
$\phi_{\bs,\bell}$ on $\II^{|\nu_\bs|}$
such that 
\begin{align*}	
\sup_{\by \in \II^{|\nu_\bs|}}\abs{\by^{\bell}  	-	\phi_{\bs,\bell}(\by)}
\ \le \
\delta_\bs,
\end{align*}
and similarly to \eqref{W(phi_s,ell)} and \eqref{L(phi_s,ell)}, the size and depth of $\phi_{\bs,\bell}$ are bounded as 
\begin{align}  \label{W-J}
W\brac{\phi_{\bs,\bell}}
\leq
C
\brac{ 
	1 
	+ 
	\abs{\bs}_1   
	\log  
	\delta_\bs^{-1}
}
\end{align}
and
\begin{align}  \label{L-J}
L\brac{\phi_{\bs,\bell}}
\leq
C
\brac{ 1 + \log	\abs{\bs}_1   \log \delta_\bs^{-1}}.
\end{align}

Let the deep ReLU neural network $\phi_\bs$ on  $\II^{|\nu_\bs|}$ be defined by
\begin{align*} 
\phi_\bs 
:=
\sum_{\bell=\boldsymbol{0}}^\bs a_{\bell} \phi_{\bs,\bell},
\end{align*}
which is a parallelization of component networks 
$\phi_{\bs,\bell}$. 
We define $\bphi_{\Lambda(\xi)}:= (\phi_\bs)_{\bs \in \Lambda(\xi)}$ as the deep ReLU neural network realized by parallelization $\phi_\bs$, $\bs \in \Lambda(\xi)$.  

The claim (i) follows from the construction. The claim (ii) can be proven  in the same way as the proof of the claim (ii) in Theorem \ref{thm1-dnn}.
\\
\\
{\bf Step 2.} In this step, we  prove the claim (v). 

From the equality
\begin{align} \nonumber
	\Phi_{\Lambda(\xi)}(\by)
	=
	\sum_{\bs\in \Lambda(\xi)} 
	v_{\bs} 
	\sum_{\bell=\boldsymbol{0}}^{\bs} 
	a_{\bell}\phi_{\bs,\bell} (\by),
\end{align}
\eqref{S_Lmabda(xi)-J} and \eqref{delta-J}, similarly to \eqref{S_Lambda - S_Lambda^omega},  we  have that
\begin{align}
	\| S_{\Lambda(\xi)} v- \Phi_{\Lambda(\xi)} v\|_{\Ll_2(X)}
	&=
	\notag
	\norm{\sum_{\bs\in \Lambda(\xi)} v_{\bs} J_{\bs} -  \sum_{\bs \in \Lambda(\xi)} v_\bs \phi_\bs(\by) }{\Ll_2(X)}
	\\
	& \leq 
	\notag
	\sum_{\bs\in \Lambda(\xi)} 
	\norm{v_{\bs}}{X} 
	\sum_{\bell=\boldsymbol{0}}^{\bs} 
	|a_{\bell}|\delta_\bs
	\\
	&
	\leq 
	\xi^{- 1/q - 1/2}\sum_{\bs\in \Lambda(\xi)} 
	\norm{v_{\bs}}{X} 
	\le 
	C\xi^{-1/q}
	\notag
\end{align}
which together with Lemma \ref{lemma[L_2-approx-J]} and \eqref{eq-decompose-J} proves the claim (iv).
\\
\\
{\bf Step 3.} 
In this step, we  prove the claims (iii) and (iv). 
%

By  Lemma \ref{lem:parallel} and \eqref{W-J} 
the size of  $\bphi_{\Lambda(\xi)} $ is estimated as
\begin{align} \label{wflambda1-J}
	W\big(\bphi_{\Lambda(\xi)} \big)
	= 
	\sum_{\bs\in \Lambda(\xi)} \sum_{\bell=\boldsymbol{0}}^{\bs} L(\phi_{\bs,\bell}) 
	\leq 
	C
		\sum_{(\bs,\bell)\in \Lambda^*(\xi)}
	\brac{ 1+ \abs{\bs}_1\log \delta_\bs^{-1}}
\end{align}	
where $\Lambda^*(\xi)$ is given as in \eqref{lambda star}.

	From \eqref{delta-J}, the inequality $\log \xi \ge 1$ and Lemma \ref{lemma: J_s} we derive that
\begin{equation*}\label{log delta-J}
	\begin{aligned}
		\log(\delta_\bs^{-1})
		&
		\le 
		C \left((1/q + 1/2)\log\xi  + \log p_\bs(1) + \log\brac{\sum_{\bell=\boldsymbol{0}}^{\bs}|a_{\bell}|}\right)
	\\
	&
	\le 
	C \left(\log\xi  + \log p_\bs(1) + |\bs|_1\right).
	\end{aligned}
\end{equation*}
	Noting  $|\bs|_1^2 \ge |\bs|_1$ for all $\bs \in \FF$, we obtain
	\begin{equation*}\label{eq-log-Absbell}
		\begin{aligned}
			\sum_{(\bs,\bell)\in \Lambda^*(\xi)}\brac{ 1 + \abs{\bs}_1\log \delta_\bs^{-1}}
			&
			\le 
			C\Bigg(   \xi \sum_{(\bs,\bell)\in \Lambda^*(\xi)}
			\abs{\bs}_1  + \sum_{(\bs,\bell)\in \Lambda^*(\xi)}
			\abs{\bs}_1 \log p_\bs(1)
			\\
			&
			+
			\sum_{(\bs,\bell)\in \Lambda^*(\xi)}
			\abs{\bs}_1^2 \Bigg)
		\end{aligned}
	\end{equation*}
	Hence, in a way similar to the proof of \eqref{sum_{(bs,bell) in Lambda^*(xi)}} we obtain 
	\begin{align*}
		\sum_{(\bs,\bell)\in \Lambda^*(\xi)}\brac{ 
			1 
			+ 
			\abs{\bs}_1
			\log 
			\delta_\bs^{-1}
		} 
		\le 
		C\xi\log\xi,
	\end{align*}
	which combined with \eqref{wflambda1-J} gives the claim (iii).

We now prove the claim (iv).
In same way as the proof of \eqref{L-flambda2} we derive that 
the depth of  $\bphi_{\Lambda(\xi)} $ is bounded as
\begin{align} 	\label{L-flambda2-J}
	L\big(\bphi_{\Lambda(\xi)} \big)
	&
	\leq 
	C
	\max_{\bs\in \Lambda(\xi)}
	\brac{ 
		\log \abs{\bs}_1
	}\max_{(\bs,\bell)\in \Lambda^*(\xi)}
	\brac{ 
		\log  \delta_\bs^{-1}
	}.
\end{align}	
By Lemma \ref{lemma:m_1 < log xi} we  obtain
\begin{align} \label{max log}
	\max_{\bs\in \Lambda(\xi)}\brac{ \log \abs{\bs}_1}
	\le 
	C \log \log \xi.
\end{align}	
We have by \eqref{log delta-J} that
\begin{equation}\label{L-eq-log-Absbell-J}
	\begin{aligned}
		\max_{(\bs,\bell)\in \Lambda^*(\xi)}\brac{ \delta_\bs^{-1}}
		&
		\le 
		C \brac{
			\log \xi
			+ \max_{\bs\in \Lambda(\xi)}
			\log p_\bs(1)
			+
			\max_{\bs\in \Lambda(\xi)}
		\log\brac{\sum_{\bell=\boldsymbol{0}}^{\bs}|a_{\bell}|}}.
	\end{aligned}
\end{equation}
For the second  on the right-hand side, we have by the well-known inequality $\log p_\bs(1) \le |\bs|_1$ and Lemma \ref{lemma:m_1 < log xi},
\begin{equation}\label{L-eq-first-term-J}
	\max_{\bs\in \Lambda(\xi)} \log p_\bs(1) 
	\leq 
	\max_{\bs\in \Lambda(\xi)} \abs{\bs}_1  \leq C\log \xi. 
\end{equation}
Now we turn to the third term in \eqref{L-eq-log-Absbell-J}. 
From Lemma \ref{lemma: J_s} it follows that
\begin{align*}
	\log\brac{\sum_{\bell=\boldsymbol{0}}^{\bs}|a_{\bell}|}
	&
	\le 
	\log \brac{(6K_{a,b})^{|\bs|_1}}
	\le 
	C |\bs|_1.
\end{align*}
Hence, again by Lemma  \ref{lemma: J_s}
\[
\max_{\bs\in \Lambda(\xi)}
	\log\brac{\sum_{\bell=\boldsymbol{0}}^{\bs}|a_{\bell}|} 
\le 
C\max_{\bs\in \Lambda(\xi)}\abs{\bs}_1
\le 
C \log \xi. 
\]
This together with \eqref{L-flambda2-J}--\eqref{L-eq-first-term-J}
yields the claim (iv).

The proof of Theorem \ref{L-thm2-dnn} is complete. 	
	\hfill
\end{proof}	

\subsection{Application to parameterized elliptic PDEs with affine inputs}
\label {affine input}

We now apply Theorem \ref{L-thm2-dnn} to approximation of the solution $u(\by)$ to the parameterized elliptic PDEs \eqref{ellip} with affine inputs \eqref{affine}. 
\begin{theorem}\label{PDE-affine-thm2-dnn}
	Let $0 < q <\infty$, $c_k^{a,b}$ be defined as in \eqref{eq-cabk} and let $(\delta_j)_{j\in \NN}$ be a sequence of numbers strictly larger than 1 such that $(\delta_j^{-1}) _{j \in \NN} \in \ell_q(\NN)$. 	Let $\bar a\in L_\infty(D)$  and ${\rm ess} \inf \bar a>0$.
	Assume that there exists a sequence of positive numbers $(\rho_j) _{j \in \NN}$   such that $c_k^{a,b}\rho_j^{-k}<\delta_j^{-k}$, $k,j\in \NN$,  and
	\begin{equation} \label{eq-affine-condition}
	\left \| \frac{\sum _{j \in \NN} \rho_j|\psi_j|}{\bar a} \right \|_{L_\infty(D)} 
	< 1 \;.
	\end{equation} 
	Then
	for every integer $n > 1$, we can  construct  a deep ReLU neural network $\bphi_{\Lambda(\xi_n)}:= (\phi_\bs)_{\bs \in \Lambda(\xi_n)}$ on 
	$\RR^m$ with $m:=\lfloor K \frac{n}{\log n} \rfloor$ for some positive constant $K$, having the following properties. 
	\begin{itemize}
			\item [{\rm (i)}]
		The deep ReLU neural network $\bphi_{\Lambda(\xi_n)}$ is independent of $u$;
		\item [{\rm (ii)}] 
		The input and output dimensions of $\bphi_{\Lambda(\xi_n)}$ are at most $m$;	
		\item [{\rm (iii)}] 
		$W\big(\bphi_{\Lambda(\xi_n)}\big) \le n $;
		\item [{\rm (iv)}] 
		$L\big(\bphi_{\Lambda(\xi_n)}\big)\le  C \log n \, \log\log n$;
		\item [{\rm (v)}]The  components $\phi_{\bs}$, $\bs\in \Lambda(\xi_n)$, of $\bphi_{\Lambda(\xi_n)}$ are deep ReLU neural networks on 
		$\RR^{|\nu_\bs|}$.
		If $\Phi_{\Lambda(\xi_n)}u$  is	defined as in Theorem \ref{thm1-dnn}(iv) with replacing $\RRi$ by $\IIi$, then the  approximation of $u$ by  $\Phi_{\Lambda(\xi_n)}u$  gives the error estimates
		$$
		\| u- \Phi_{\Lambda(\xi_n)} u \|_{\Ll_2(V)} 
			\leq C m^{-1/q}
		\leq C \left(\frac{n}{\log n}\right)^{-1/q}.
		$$
	\end{itemize}		
 	Here the positive constants  $C$  are independent of $u$ and $n$.
\end{theorem}

\begin{proof}
	It has been proven in \cite{BCM17} that under the assumptions of the theorem,
	for the sequence $(\sigma_\bs)_{\bs \in \FF}$ given as in  \eqref{sigma_bs},
	\begin{equation} \nonumber
	\sum_{\bs\in\FF} (\sigma_\bs \|u_\bs\|_V)^2 <\infty.
	\end{equation}
This means that Assumption B holds for $v = u$ with $X= V$. Hence, applying Theorem \ref{L-thm2-dnn} to $u$, we prove the theorem.  
	\hfill
\end{proof}


We next discuss the approximation by deep ReLU neural networks for parametric elliptic PDEs with affine inputs and error measured in the uniform norm of $L_\infty(\IIi,V)$ by using  $m$-term truncations of the Taylor gpc expansion of $u$.
This problem has been studied in \cite{SZ19}  for the particular case of the uniform probability measure $\nu_{0,0}$.

If for the sequence $(\rho_j)_{j\in \NN}$ of numbers strictly larger than 1 we have the condition \ref{eq-affine-condition} and if $(\rho_j^{-1})_{j\in \NN}\in \ell_q(\NN)$ for some $0<q<2$, then the solution $u$ to the parameterized elliptic PDEs \eqref{ellip} with affine inputs \eqref{affine}  can be decomposed in the Taylor gpc expansion  

\begin{equation*} \label{T-series}
u = \sum_{\bs\in\FF} t_\bs \by^\bs, \qquad t_\bs = \frac{1}{\bs!}\partial^\bs u(\boldsymbol{0})
\end{equation*}
with
\begin{equation*}  
	\left(\sum_{\bs\in\FF} (\sigma_\bs \|t_\bs\|_V)^2\right)^{1/2} \ \le C \ <\infty,
\end{equation*}
where
\begin{equation*} 
	\sigma_\bs  :=   \prod_{j \in \NN} \rho_j^{s_j},
\end{equation*}
see \cite[Theorem 2.1]{BCM17}. Moreover, the sequence $(\|t_\bs\|_V)_{\bs\in \FF}$ is $\ell_p$-summable with $p=\frac{2q}{2+q} < 1$. We define
	\begin{equation*} 
	S_{\Lambda(\xi)} v 
	:= \ 
	\sum_{\bs\in {\Lambda(\xi)}} t_\bs \by_\bs,
\end{equation*}
where $\Lambda(\xi)$ is given by the formula \eqref{Lambda(xi)}. 
The following theorem is an improvement   of \cite[Theorem 3.9]{SZ19}.
	 
\begin{theorem}\label{PDE-affine-taylor-thm2-dnn}
Let $\bar a\in L_\infty(D)$  and ${\rm ess} \inf \bar a>0$.
Assume that there exists an increasing sequence $(\rho_j) _{j \in \NN}$ of  numbers strictly larger than 1 such that 
the sequence $(\rho_j^{-1}) _{j \in \NN} \in \ell_{q}(\NN)$ for some $q$ with $0< q < 2$, and there holds the condition \eqref{eq-affine-condition}.	
	Then
	for every integer $n \ge 4$, we can  construct  a deep ReLU neural network $\bphi_{\Lambda(\xi_n)}:= (\phi_\bs)_{\bs \in \Lambda(\xi_n)}$ on 
	$\RR^m$ with $m:=\left\lfloor K \frac{n}{\log n} \right\rfloor$ for some positive constant $K$, having the following properties. 
	\begin{itemize}
			\item [{\rm (i)}]
		The deep ReLU neural network $\bphi_{\Lambda(\xi_n)}$ is independent of $u$;
		\item [{\rm (ii)}] 
		The input and output dimensions of $\bphi_{\Lambda(\xi_n)}$ are at most $m$;	
		\item [{\rm (iii)}] 
		$W\big(\bphi_{\Lambda(\xi_n)}\big) \le n $;
		\item [{\rm (iv)}] 
		$L\big(\bphi_{\Lambda(\xi_n)}\big)\le  C \log n \log\log n$;
	\item [{\rm (v)}]
		The  components $\phi_{\bs}$, $\bs\in \Lambda(\xi_n)$, of $\bphi_{\Lambda(\xi_n)}$ are deep ReLU neural networks on 
			$\RR^{|\nu_\bs|}$.
			If $\Phi_{\Lambda(\xi_n)}u$  is	defined as in Theorem \ref{thm1-dnn}(iv) with replacing $\RRi$ by $\IIi$ and $u_\bs$ by $t_\bs$, then the  approximation of $u$ by  $\Phi_{\Lambda(\xi_n)}u$  gives the error estimates
		$$
	\| u- \Phi_{\Lambda(\xi_n)}u\|_{L_\infty(\IIi, V)} 
		\leq C m^{-(1/q - 1/2)}
	\leq C \left(\frac{n}{\log n}\right)^{-(1/q - 1/2)}.
	$$
	\end{itemize}
Here the positive constants $C$ are independent of $u$ and $n$.
\end{theorem}
\begin{proof}
This theorem can be proven in a way similar  to the proof of Theorem \ref{PDE-affine-thm2-dnn}.  Let us give a brief proof.	 	
Given $\xi \geq 3$, we have  the Cauchy-Schwarz  inequality and Lemma \ref{bcdmJ}  that 
\begin{align} \label{u-S_Taylor}
\| u- S_{\Lambda(\xi)} u\|_{L_\infty(\IIi,V)} 
&
\leq \sum_{\bs\not\in {\Lambda(\xi)}}\| t_\bs \|_V
\leq
\brac{ 
	\sum_{\sigma_{\bs}> \xi^{1/q}}
	(\sigma_{\bs}\norm{t_{\bs}}{V})^2
}^{1/2}
\brac{ 
	\sum_{\sigma_{\bs}> \xi^{1/q}}
	\sigma_{\bs}^{-q} \sigma_{\bs}^{-(2-q)}
}^{1/2}
\\
&
\leq 
C
\xi^{-(1/q-1/2)} 
\brac{ 
	\sum_{\bs\in \FF}
	\sigma_{\bs}^{-q}
}^{1/2} 
\leq 
C
\xi^{-(1/q-1/2)}.
\notag
\end{align}	

Put 
$\delta := \xi^{-(1/q-1/2)}$.  For every $\bs \in \Lambda(\xi)$, by Lemma \ref{lem-product-varphi}  
there exists a deep  ReLU neural network
$\phi_{\bs}$ on $\II^{|\nu_\bs|_1}$
such that
\begin{align*} 	
	\sup_{\by \in \II^{|\nu_\bs|_1}}\abs{\by^{\bs}  	-	\phi_{\bs}(\by)}
	\ \le \
	\delta,
\end{align*}
and the size and depth of $\phi_{\bs}$ are bounded as 
\begin{align*} 
	W\brac{\phi_{\bs}}
	\leq
	C
	\brac{ 
		1 
		+ 
		\abs{\bs}_1   
		\log  
		\delta^{-1}
	} \leq  C\big(1+ |\bs|_1\log\xi \big)
\end{align*}
and
\begin{align*} 
	L\brac{\phi_{\bs}}
	\leq
	C
	\brac{ 1 + \log	\abs{\bs}_1   \log \delta^{-1}} \leq C\big(1+ \log|\bs|_1\log\xi \big).
\end{align*} 
We define $\bphi_{\Lambda(\xi)}:= (\phi_\bs)_{\bs \in \Lambda(\xi)}$ as the deep ReLU neural network realized on $\IIm$ by parallelization of $\phi_\bs$, $\bs \in \Lambda(\xi)$. Consider the approximation of $u$ by 
$$
\Phi_{\Lambda(\xi)}u:= \sum_{\bs \in \Lambda(\xi)}t_\bs \phi_\bs.
$$
Then by the inclusion $(\|t_\bs\|_V)_{\bs\in \FF}\in \ell_p(\FF)$, $p\in (0,1)$ and \eqref{u-S_Taylor}, we have
\begin{align} 
	\| u- \Phi_{\Lambda(\xi)}u\|_{L_\infty(\IIi, V)}&\leq \|u- S_{\Lambda(\xi)} u\|_{L_\infty(\IIi, V)}+	
	\norm{S_{\Lambda(\xi)} u - \Phi_{\Lambda(\xi)} u}{L_\infty(\IIi, V)}
	\notag
	\\
	& \leq  \ C\xi^{-(1/q-1/2)}+
	\sum_{\bs\in {\Lambda(\xi)}} \|t_\bs\|_V \norm{\by_\bs - \phi_\bs}{L_\infty(\IIi, V)}
	 \label{u-Phi-T} 
	\\
& \leq  \ C\xi^{-(1/q-1/2)}+ C\xi^{-(1/q-1/2)}
\sum_{\bs\in {\Lambda(\xi)}} \|t_\bs\|_V \leq 
C\xi^{-(1/q-1/2)},
\notag
\end{align}  
where the positive constants $C$ may be different and are independent of $u$ and $\xi$. By the construction of $\bphi_{\Lambda(\xi)}$ we have
\begin{equation} \label{W-T}
\begin{aligned}
W(\bphi_{\Lambda(\xi)}) & \leq \sum_{\bs \in \Lambda(\xi)} \leq	W\brac{\phi_{\bs}} \leq \sum_{\bs \in \Lambda(\xi)}C\big(1+ |\bs|_1\log\xi \big) \leq C \bigg(|\Lambda(\xi)|+ \log\xi \sum_{\sigma_\bs^q\leq \xi} p_\bs(1)\bigg) 
\\
&
\leq C \bigg(|\Lambda(\xi)|+ \log\xi \sum_{\sigma_\bs^q\leq \xi} p_\bs(1)\sigma_\bs^{-q}\sigma_\bs^q\bigg) \leq C \xi \log \xi
\end{aligned}
\end{equation}
where in the last estimate we used Lemmata \ref{lemma: |s|_1, |s|_0}(i) and \ref{bcdmJ}. Similarly, we have
\begin{equation} \label{L-T}
L(\bphi_{\Lambda(\xi)}) \leq \max_{\bs \in \Lambda(\xi)}	L\brac{\phi_{\bs}} \leq C\max_{\bs \in \Lambda(\xi)}\big(1+ \log|\bs|_1\log\xi \big) \leq C\log \xi \log\log\xi, 
\end{equation}
see Lemma \ref{lemma:m_1 < log xi}. Now following argument at the end of the proof of Theorem \ref{thm-PDE-lognormal-dnn}, by using \eqref{u-Phi-T}--\eqref{L-T} we obtain the existence of a number $\xi_n$ for a given $n \ge 4$ for which there hold the claims (i)--(v) in the theorem.
\hfill
\end{proof}

\bigskip
\noindent
{\bf Acknowledgments.}  
A part of this work was done when  the authors were working at the Vietnam Institute for Advanced Study in Mathematics (VIASM). They would like to thank  the VIASM  for providing a fruitful research environment and working condition.

\bibliographystyle{abbrv}
\bibliography{AllBib}

\begin{thebibliography}{10}

\bibitem{AlNo20}
M.~Ali and A.~Nouy.
\newblock {Approximation of smoothness classes by deep ReLU networks}.
\newblock {\em arXiv:2007.15645}, 2020.

\bibitem{ABMM17}
R.~Arora, A.~Basu, P.~Mianjy, and A.~Mukherjee.
\newblock Understanding deep neural networks with rectified linear units.
\newblock {\em Electronic Colloquium on Computational Complexity}, Report No.
  98, 2017.

\bibitem{BCDC17}
M.~Bachmayr, A.~Cohen, D.~{D\~ ung}, and C.~Schwab.
\newblock {Fully discrete approximation of parametric and stochatic elliptic
  PDEs}.
\newblock {\em SIAM J. Numer. Anal.}, 55:2151--2186, 2017.

\bibitem{BCDM17}
M.~Bachmayr, A.~Cohen, R.~DeVore, and G.~Migliorati.
\newblock {Sparse polynomial approximation of parametric elliptic PDEs. Part
  II: lognormal coefficients}.
\newblock {\em ESAIM Math. Model. Numer. Anal.}, 51:341 -- 363, 2017.

\bibitem{BCM17}
M.~Bachmayr, A.~Cohen, and G.~Migliorati.
\newblock {Sparse polynomial approximation of parametric elliptic PDEs. Part I:
  affine coefficients}.
\newblock {\em ESAIM Math. Model. Numer. Anal.}, 51:321--339, 2017.

\bibitem{Bar91}
A.~R. Barron.
\newblock Complexity regularization with application to artificial neural
  networks.
\newblock {\em In Nonparametric Functional Estimation and Related Topics. NATO
  ASI Series (Series C: Mathematical and Physical Sciences)}, 335:561--576,
  1991.

\bibitem{CoDe15a}
A.~Cohen and R.~DeVore.
\newblock {Approximation of high-dimensional parametric PDEs}.
\newblock {\em Acta Numer.}, 24:1--159, 2015.

\bibitem{CDS10}
A.~Cohen, R.~DeVore, and C.~Schwab.
\newblock {Convergence rates of best $N$-term Galerkin approximations for a
  class of elliptic sPDEs}.
\newblock {\em Found. Comput. Math.}, 9:615--646, 2010.

\bibitem{CDS11}
A.~Cohen, R.~DeVore, and C.~Schwab.
\newblock {Analytic regularity and polynomial approximation of parametric and
  stochastic elliptic PDE's}.
\newblock {\em Anal. Appl.}, 9:11--47, 2011.

\bibitem{Cyben89}
G.~Cybenko.
\newblock {Approximation by superpositions of a sigmoidal function}.
\newblock {\em Math. Control. Signals, Syst.}, 2:303--314, 1989.

\bibitem{Dung22}
D.~{D\~ung}.
\newblock {Collocation approximation by deep neural ReLU networks for
  parametric elliptic PDEs with lognormal inputs}.
\newblock {\em arXiv:2111.05504}.

\bibitem{Dung19}
D.~{D\~ung}.
\newblock {Linear collocation approximation for parametric and stochastic
  elliptic PDEs}.
\newblock {\em Mat. Sb.}, 210:103--227, 2019.

\bibitem{Dung21}
D.~{D\~ung}.
\newblock {Sparse-grid polynomial interpolation approximation and integration
  for parametric and stochastic elliptic PDEs with lognormal inputs}.
\newblock {\em ESAIM Math. Model. Numer. Anal.}, 55:1163--1198, 2021.

\bibitem{DN20}
D.~{D\~ung} and V.~K. Nguyen.
\newblock {Deep ReLU neural networks in high-dimensional approximation}.
\newblock {\em Neural Netw.}, 142:619--635, 2021.

\bibitem{DDK21}
D.~{D\~ung}, V.~K. Nguyen, and D.~T. Pham.
\newblock {Deep ReLU neural network approximation of parametric and stochastic
  elliptic PDEs with lognormal inputs}.
\newblock {\em arXiv:2111.05854}, 2021.

\bibitem{DNSZ22}
D.~{D\~ung}, V.~K. Nguyen, C.~Schwab, and J.~Zech.
\newblock {Analyticity and sparsity in uncertainty quantification for PDEs with
  Gaussian random field inputs}.
\newblock {\em arXiv:2201.01912}, 2022.

\bibitem{DDF.19}
I.~Daubechies, R.~DeVore, S.~Foucart, B.~Hanin, and G.~Petrova.
\newblock {Nonlinear approximation and (Deep) ReLU networks}.
\newblock {\em Constr. Approx.}, 55:127–172, 2022.

\bibitem{DHP21}
R.~DeVore, B.~Hanin, and G.~Petrova.
\newblock {Neural network approximation}.
\newblock {\em Acta Numer.}, pages 327--444, 2021.

\bibitem{EWa18}
W.~E and Q.~Wang.
\newblock {Exponential convergence of the deep neural network approximation for
  analytic functions}.
\newblock {\em Sci. China Math.}, 61:1733--1740, 2018.

\bibitem{EGJS18}
D.~Elbr{\"a}chter, P.~Grohs, A.~Jentzen, and C.~Schwab.
\newblock {DNN expression rate analysis of high-dimensional PDEs: application
  to option pricing}.
\newblock {\em Constr. Approx.}, 55:3–71, 2022.

\bibitem{EST18}
O.~G. Ernst, B.~Sprungk, and L.~Tamellini.
\newblock Convergence of sparse collocation for functions of countably many
  {G}aussian random variables (with application to elliptic {PDE}s).
\newblock {\em SIAM J. Numer. Anal.}, 56(2):877--905, 2018.

\bibitem{Funa90}
K.-I. Funahashi.
\newblock {Approximate realization of identity mappings by three-layer neural
  networks}.
\newblock {\em Electron. Commun. Jpn 3}, 73:61--68, 1990.

\bibitem{GPR.21}
M.~Geist, P.~C. Petersen, M.~Raslan, R.~Schneider, and G.~Kutyniok.
\newblock {Numerical solution of the parametric diffusion equation by deep
  neural networks}.
\newblock {\em J. Sci. Comput.}, 88,22, 2021.

\bibitem{GS20}
L.~Gonon and C.~Schwab.
\newblock {Deep ReLU network expression rates for option prices in
  high-dimensional, exponential L{\'e}vy models}.
\newblock {\em Finance Stoch.}, 25:615--657, 2021.

\bibitem{GS21}
L.~Gonon and C.~Schwab.
\newblock {Deep ReLU neural network approximation for stochastic differential
  equations with jumps}.
\newblock Technical Report 2021-08, Seminar for Applied Mathematics, ETH
  Z{\"u}rich, 2021.

\bibitem{GKNV21}
R.~Gribonval, Kutyniok, M.~Nielsen, and F.~Voigtl{\"a}nder.
\newblock Approximation spaces of deep neural networks.
\newblock {\em Constr. Approx.}, 55:259--367, 2022.

\bibitem{GH21}
P.~Grohs and L.~Herrmann.
\newblock {Deep neural network approximation for high-dimensional elliptic PDEs
  with boundary conditions}.
\newblock {\em IMA J. Numer. Anal.}, 2021,
  https://doi.org/10.1093/imanum/drab031.

\bibitem{GJS19}
P.~Grohs, A.~Jentzen, and D.~Salimova.
\newblock {Deep neural network approximations for Monte Carlo algorithms}.
\newblock {\em arXiv:1908.10828}, 2019.

\bibitem{GPEB21}
P.~Grohs, D.~Perekrestenko, D.~Elbrachter, and H.~Bolcskei.
\newblock {Deep neural network approximation theory}.
\newblock {\em IEEE Trans. Inf. Theory}, 67:2581--2623, 2021.

\bibitem{GKP20}
I.~G\"uhring, G.~Kutyniok, and P.~Petersen.
\newblock {Error bounds for approximations with deep ReLU neural networks in
  $W^{s,p}$ norms}.
\newblock {\em Anal. Appl. (Singap.)}, 18:803--859, 2020.

\bibitem{HSZ20}
L.~Herrmann, C.~Schwab, and J.~Zech.
\newblock {Deep neural network expression of posterior expectations in Bayesian
  PDE inversion}.
\newblock {\em Inverse Problems}, 36:125011, 2020.

\bibitem{HS65}
E.~Hewitt and K.~Stromberg.
\newblock {\em {Real and Abstract Analysis}}.
\newblock Springer, 1965.

\bibitem{HoSc14}
V.~Hoang and C.~Schwab.
\newblock {$N$-term Galerkin Wiener chaos approximation rates for elliptic PDEs
  with lognormal Gaussian random inputs}.
\newblock {\em Math. Models Methods Appl. Sci.}, 24:797--826, 2014.

\bibitem{HSW89}
K.~Hornik, M.~Stinchcombe, and H.~White.
\newblock {Multilayer feedforward networks are universal approximators}.
\newblock {\em Neural Netw.}, 2:359--366, 1989.

\bibitem{HJKN19}
M.~Hutzenthaler, A.~Jentzen, T.~Kruse, and T.~A. Nguyen.
\newblock A proof that rectified deep neural networks overcome the curse of
  dimensionality in the numerical approximation of semilinear heat equations.
\newblock {\em SN Partial Differ. Equ. Appl.}, 1, 2020.

\bibitem{JSW18}
A.~Jentzen, D.~Salimova, and T.~Welti.
\newblock {A proof that deep artificial neural networks overcome the curse of
  dimensionality in the numerical approximation of Kolmogorov partial
  differential equations with constant diffusion and nonlinear drift
  coefficients}.
\newblock {\em Commun. Math. Sci.}, 19:1167--1205, 2021.

\bibitem{KPRS21}
G.~Kutyniok, P.~C. Petersen, M.~Raslan, and R.~Schneider.
\newblock {A theoretical analysis of deep neural networks and parametric PDEs}.
\newblock {\em Constr. Approx.}, 55:73--125, 2022.

\bibitem{Lu07}
D.~S. Lubinsky.
\newblock {A survey of weighted polynomial approximation with exponential
  weights}.
\newblock {\em Surveys in Approximation Theory}, 3:1--105, 2007.

\bibitem{Mha96}
H.~N. Mhaskar.
\newblock {Neural networks for optimal approximation of smooth and analytic
  functions}.
\newblock {\em Neural Comput.}, 8:164--177, 1996.

\bibitem{MoDu19}
H.~Montanelli and Q.~Du.
\newblock {New error bounds for deep ReLU networks using sparse grids}.
\newblock {\em SIAM J. Math. Data Sci.}, 1:78--92, 2019.

\bibitem{MPCB14}
G.~Mont\'ufar, R.~Pascanu, K.~Cho, and Y.~Bengio.
\newblock On the number of linear regions of deep neural networks.
\newblock {\em In Advances in neural information processing systems}, pages
  2924--2932, 2014.

\bibitem{OSZ21}
J.~A.~A. Opschoor, C.~Schwab, and J.~Zech.
\newblock {Exponential ReLU DNN expression of holomorphic maps in high
  dimension}.
\newblock {\em Constr. Approx.}, 55:537--582, 2022.

\bibitem{Pet20}
P.~C. Petersen.
\newblock {Neural network theory}.
\newblock {\em Available at http://pc-petersen.eu/Neural\_Network\_Theory.pdf}.

\bibitem{PeVo18}
P.~C. Petersen and F.~Voigtl{\"a}nder.
\newblock {Optimal approximation of piecewise smooth functions using deep ReLU
  neural networks}.
\newblock {\em Neural Netw.}, 108:296--330, 2018.

\bibitem{SZ19}
C.~Schwab and J.~Zech.
\newblock {Deep learning in high dimension: Neural network expression rates for
  generalized polynomial chaos expansions in UQ}.
\newblock {\em Anal. Appl. (Singap.)}, 17:19--55, 2019.

\bibitem{SZ2021}
C.~Schwab and J.~Zech.
\newblock Deep learning in high dimension: Neural network expression rates for
  analytic functions in $l^2(\mathbb{R}^d,\gamma_d)$.
\newblock 2021.

\bibitem{SS18}
J.~Sirignano and K.~Spiliopoulos.
\newblock {DGM: A deep learning algorithm for solving partial differential
  equations }.
\newblock {\em J. Comput. Phys.}, 375:1339--1364, 2018.

\bibitem{Suzu18}
T.~Suzuki.
\newblock {Adaptivity of deep ReLU network for learning in Besov and mixed
  smooth Besov spaces: optimal rate and curse of dimensionality}.
\newblock {\em International Conference on Learning Representations}, 2019.

\bibitem{Te15}
M.~Telgarsky.
\newblock Representation benefits of deep feedforward networks.
\newblock {\em arXiv:1509.08101}, 2015.

\bibitem{Te16}
M.~Telgrasky.
\newblock Benefits of depth in neural nets.
\newblock {\em In Proceedings of the JMLR: Workshop and Conference Proceedings,
  New York, NY, USA}, 49:1--23, 2016.

\bibitem{TB18}
R.~Tripathy and I.~Bilionis.
\newblock {Deep UQ: Learning deep neural network surrogate models for high
  dimensional uncertainty quantification}.
\newblock {\em J. Comput. Phys.}, 375:565--588, 2018.

\bibitem{Ya17a}
D.~Yarotsky.
\newblock {Error bounds for approximations with deep ReLU networks}.
\newblock {\em Neural Netw.}, 94:103--114, 2017.

\bibitem{Ya18}
D.~Yarotsky.
\newblock {Optimal approximation of continuous functions by very deep ReLU
  networks}.
\newblock {\em Proc. Mach. Learn. Res.}, 75:1--11, 2018.

\bibitem{ZDS19}
J.~Zech, D.~{D\~ung}, and C.~Schwab.
\newblock {Multilevel approximation of parametric and stochastic PDES}.
\newblock {\em {Math. Models Methods Appl. Sci.}}, 29:1753--1817, 2019.

\end{thebibliography}

\end{document}